\documentclass[11pt]{amsart}  
\usepackage{amsmath,amssymb,amsthm,amsfonts}
\usepackage{mathrsfs}
\usepackage{mathtools}
\usepackage{xcolor}
\usepackage[all,cmtip]{xy}

\usepackage{geometry}
\usepackage[style=alphabetic]{biblatex} 
\usepackage{hyperref}
\hypersetup{
    colorlinks,
    linkcolor={black},
    citecolor={black},
    urlcolor={black}
}

\theoremstyle{plain} 
\newtheorem{thm}{Theorem}[section]

\newtheorem{cor}[thm]{Corollary}

\newtheorem{prop}[thm]{Proposition}
\newtheorem{defn}[thm]{Definition}
\theoremstyle{remark}
\newtheorem{rem}[thm]{Remark}

\newcommand{\R}{\mathbb{R}}
\newcommand{\C}{\mathbb{C}}
\newcommand{\Q}{\mathbb{Q}}
\newcommand{\Fp}{\mathbb{F}_p}
\newcommand{\E}{\mathop{\mathbb{E}}}

\makeatletter
\newcommand{\vast}{\bBigg@{4}}
\newcommand{\Vast}{\bBigg@{5}}
\makeatother

\newcommand{\bn}{\mathbf{n}}
\newcommand{\bm}{\mathbf{m}}
\newcommand{\bx}{\mathbf{x}}
\newcommand{\bh}{\mathbf{h}}

\newcommand{\bz}{\mathbf{z}}
\newcommand{\dprime}{\prime\prime}

\title[Nonlinear Szemer\'{e}di theorem for corners]{Uniform nonlinear Szemer\'{e}di theorem for corners in finite fields}

\author{Zi Li Lim}
\address{Department of Mathematics, UCLA, Los Angeles, CA 90095, USA}
\email{zililim@math.ucla.edu}

\begin{document}

\maketitle

\begin{abstract}
Let $P(t),Q(t)\in \Q(t)$ be rational functions such that $P(t),Q(t)$ and the constant function $1$ are linearly independent over $\Q$, we prove an asymptotic formula for the number of the corner configurations $(x_1,x_2),(x_1+P(y),x_2),(x_1,x_2+Q(y))$ in the subsets of $\Fp^2$.
\end{abstract}

\section{Introduction}\label{sec: intro}
\subsection{Introduction}

The monumental paper of Bergelson and Leibman \cite{bergelson-leibman} provides a satisfactory qualitative understanding of the sets that do not contain polynomial progressions, i.e., configurations of the form
\begin{equation*}
    x, x+P_1(y), x+P_2(y),...,x+P_l(y)
\end{equation*}
where $P_1,P_2,...,P_l$ are polynomials. Meanwhile, the quantitative understanding of this subject is still formative. Recent breakthroughs of Peluse \cite{peluse,peluse-polynomial} and of Peluse and Prendiville \cite{peluse-prendiville} bring major progress on quantitative estimates for sets lacking one dimensional polynomial progressions, but the theory for higher dimensional polynomial progressions is far from complete.

There are significant new difficulties for higher dimensional progressions, which could be found in Peluse's paper on L-shaped configurations \cite{peluse-L}. Further evidence regarding high dimensional difficulties are recorded in Austin's work \cite{austin}.

In this paper, we consider the following two dimensional configurations in finite fields $\Fp$
\begin{equation*}
    (x_1,x_2),(x_1+P(y),x_2),(x_1,x_2+Q(y))
\end{equation*}
where $P,Q$ are rational functions. We call these configurations \emph{corners generated by $P,Q$}. Here and throughout, we use the notation $\E$ for expectations, and it will exclude the poles when the summands of $\E$ involve rational functions. Besides, we will often abbreviate $\E_{x\in \Fp}$ as $\E_{x}$. Our main result is the following asymptotic formula for counting operator for corners generated by $P,Q$.

\begin{thm}\label{thm: main}
    Let $P(t),Q(t)\in \Q(t)$ be rational functions such that $P(t),Q(t)$ and the constant function $1$ are linearly independent over $\Q$, then we have the asymptotic formula
    \begin{multline*}
        \E_{x_1,x_2,y} f_0(x_1,x_2) f_1(x_1+P(y),x_2) f_2(x_1,x_2+Q(y))\\
        =\E_{x_1,x_2} \Bigl(f_0(x_1,x_2) \E_{a}f_1(a,x_2) \E_{b}f_2(x_1,b)\Bigr)+O_{P,Q}\bigl(p^{-\frac{1}{40960}}\bigr)
    \end{multline*}
    for all $1$-bounded functions $f_0,f_1,f_2:\Fp^2\longrightarrow \C$.
\end{thm}

\begin{rem}
    We have not attempted to optimize the power-saving exponent $1/40960$. The precise numerical value of the power-saving exponent can be ignored for a first reading.
\end{rem}

\begin{rem}
    With more effort, it can be shown that the implied constant of the error term in Theorem \ref{thm: main} actually only depends on the degrees of $P,Q$, where we define the degree of a rational function to be the sum of the degrees of the numerator and denominator. Also, we refer to Section \ref{sec: convention} for more details regarding our asymptotic notation convention.
\end{rem}

As usual, one could derive an upper bound for subsets $A\subset \Fp^2$ that do not contain corners generated by $P,Q$. To see that, let $f_0,f_1,f_2$ be the characteristic function $1_{A}$, and note that
\begin{equation*}
    \E_{x_1,x_2} \Bigl(1_{A}(x_1,x_2) \E_{a}1_{A}(a,x_2) \E_{b}1_{A}(x_1,b)\Bigr)\gg \Bigl(\E_{x_1,x_2} 1_{A}(x_1,x_2)\Bigr)^3
\end{equation*}
by the pigeonhole principle (see \cite[Lemma 3.1]{han-lacey-yang} for details). This leads to the following power-saving bound.

\begin{cor}\label{cor: main}
    Let $P(t),Q(t)\in \Q(t)$ be rational functions such that $P(t),Q(t)$ and the constant function $1$ are linearly independent over $\Q$. Let $A$ be a subset of $\Fp^2$ with density $\delta=\frac{|A|}{p^2}\gg p^{-\frac{1}{122880}}$, then $A$ contains corners generated by $P,Q$. In fact, there are plenty of corners generated by $P,Q$ in $A$, that is
    \begin{equation*}
        \#\{(x_1,x_2,y)\in \Fp^3:\text{The corner}\; (x_1,x_2),(x_1+P(y),x_2),(x_1,x_2+Q(y))\; \text{is in}\; A\}\gg p^3\delta^3.
    \end{equation*}
\end{cor}

As a remark, the linear independence assumption of $P,Q$ and constant function $1$ is not restrictive, it includes the cases that $P,Q$ are linearly independent polynomials with zero constant terms, which are the common assumptions for quantitative polynomial Szemer\'{e}di theorem. For example, the configurations we consider include
\begin{equation*}
    (x_1,x_2),(x_1+y,x_2),(x_1,x_2+y^2)
\end{equation*}
and
\begin{equation*}
    (x_1,x_2),(x_1+y^3,x_2),(x_1,x_2+y^3-y^2+y)
\end{equation*}
and even strange examples like
\begin{equation*}
    (x_1,x_2),\left(x_1+\frac{y^2}{y^7-5y^3},x_2\right),\left(x_1,x_2+y^{17}+\frac{1}{y^{13}+19}\right).
\end{equation*}

Compared to previous results, Theorem \ref{thm: main} and Corollary \ref{cor: main} is new in the following two aspects.
\begin{itemize}
\item Qualitatively, Theorem \ref{thm: main} and Corollary \ref{cor: main} allows $P,Q$ to be rational functions while the previous results are only applicable to polynomials.
    \item Quantitatively, the power-saving exponent $1/122880$ is an improvement upon previous results when $P,Q$ are polynomials and their degrees are the same and greater than $2$. Moreover, the power-saving exponent $1/122880$ is uniform in the degrees of $P,Q$; it does not decay exponentially as the degrees grow.    
\end{itemize}

To be more specific, previously Han, Lacey and Yang \cite{han-lacey-yang} proved that if $P,Q$ are linearly independent polynomials with zero constant terms and \textbf{distinct} degree or linearly independent quadratic polynomials with zero constant terms, and $A$ is a subset of $\Fp^2$ with density $\delta\gg p^{-1/16}$, then
\begin{equation*}
        \#\{(x_1,x_2,y)\in \Fp^3:\text{The corner}\; (x_1,x_2),(x_1+P(y),x_2),(x_1,x_2+Q(y))\; \text{is in}\; A\}\gg p^3\delta^3.
    \end{equation*}
On the other hand, Kuca \cite{kuca-1,kuca-2} proved that if $P,Q$ are linearly independent polynomials with zero constant terms, then there exists a constant $c$ depending on the degrees of $P,Q$ such that
\begin{equation*}
        \#\{(x_1,x_2,y)\in \Fp^3:\text{The corner}\; (x_1,x_2),(x_1+P(y),x_2),(x_1,x_2+Q(y))\; \text{is in}\; A\}\gg p^3\delta^3
    \end{equation*}
for all subsets $A\subset \Fp^2$ with density $\delta\gg p^{-c}$. The method of Han, Lacey and Yang is Fourier analytic while the method of Kuca is based on higher order Fourier analysis and degree lowering. The power-saving exponent $c$ in Kuca's results is unspecified in his papers, but it could be made explicit in principle, although it would be fairly small due to the astronomical amount of application of the Cauchy--Schwarz inequality in degree lowering. For this same reason, the power-saving exponent $c$ would also decay quite quickly as the degrees of $P,Q$ grow. Therefore, our power-saving exponent $1/122880$ is an uniform numerical improvement only in the cases not covered by the results of Han, Lacey and Yang.

The new difficulties when moving to rational function progressions from polynomial progressions are that Weyl differencing is no longer useful and techniques based on PET induction are no longer available for obtaining Gowers norm control. The main reason why we can allow rational functions in Theorem \ref{thm: main} is that we use algebraic geomerty to bypass Weyl differencing to obtain Gowers norm control for counting operators. As a side note, the Gowers norm control we obtain is of fairly low degree, and this explains the improvement on power-saving exponent in the error term.

The approach of this paper comes from four different papers \cite{han-lacey-yang,hong-lim,kavrut-wu,kuca-2}. To elaborate, Hong and the author \cite{hong-lim} developed an algebraic geometry version of PET induction that could bypass Weyl differencing to obtain Gowers norm control for one dimensional three term rational function progressions. To generalize the algebraic geometry PET induction to the case of corners, we incorporate the method with the Fourier analytic approach due to Han, Lacey and Yang \cite{han-lacey-yang}. However, the ``naive" generalization of algebraic geometry PET induction does not work, since it would discard too much cancellation among exponential sums, and lead to degeneracy of certain varieties. Luckily, we could borrow an idea from the work of Kavrut and Wu \cite{kavrut-wu} on improving the nonlinear Roth theorem to amplify the cancellation among exponential sums and hence obtain Gowers norm control. After obtaining Gowers norm control, it is natural to apply the degree lowering pioneered by Peluse \cite{peluse} to complete the proof. In fact, we need a different version of degree lowering that works for directional Gowers norms, and this version of degree lowering was developed by Kuca \cite{kuca-2}.

To draw parallel comparisons, we remark that corners could be viewed as a two dimensional generalization of one dimensional three term progressions. For results on one dimensional three term progressions, see work of Bourgain and Chang \cite{bourgain-chang}, work of Peluse \cite{peluse-three-term,peluse}, work of Dong, Li and Sawin \cite{dong-li-sawin}, and work of Hong and the author \cite{hong-lim}. In another direction, for polynomial Szemer\'{e}di theorem for corners in integers, see work of Peluse, Prendiville and Shao \cite{peluse-prendiville-shao} and work of Kravitz, Kuca and Leng \cite{kravitz-kuca-leng}.

\subsection{Acknowledgments}

The author would like to thank his advisor Terence Tao for helpful guidance and support. The author also thanks James Leng for encouragement.

\section{Convention}\label{sec: convention}

\subsection{Notation}
Throughout, let $p$ be a prime and $\mathbb{F}_p$ be the finite field with $p$ elements. For any finite set $\mathcal{A}$ and function $f:\mathcal{A}\longrightarrow\mathbb{C}$, let $\mathbb{E}_{x\in\mathcal{A}}f(x)=\frac{1}{|\mathcal{A}|}\sum_{x\in \mathcal{A}}f(x)$ be the average of $f$ over $\mathcal{A}$. When the summation variables are in $\mathbb{F}_p$, we would often omit the subscript $\mathbb{F}_p$ in the summation notation and write $\mathbb{E}_{x}$ instead. By abuse of notation, we also use $\mathbb{E}$ to denote averaging excluding the poles of the (rational) functions being averaged. For example, if $P(t),Q(t)\in \mathbb{Q}(t)$ are rational functions and $f_0,f_1,f_2:\mathbb{F}_p^2\longrightarrow\mathbb{C}$ are functions, then
\begin{multline*}
\E_{x_1,x_2,y}f_0(x_1,x_2)f_1(x_1+P(y),x_2)f_2(x_1,x_2+Q(y))=\\\frac{1}{p^3}\sum_{x_1,x_2,y\in \mathbb{F}_p \;\text{and}\; y\neq \text{poles}}f_0(x_1,x_2)f_1(x_1+P(y),x_2)f_2(x_1,x_2+Q(y)).
\end{multline*}

Furthermore, we would often omit the subscript $P,Q$ in Vinogradov's notation $\ll,\gg$ and Landau's notation $O,o$ since the implied constants would usually depend on the rational functions $P,Q$.

Moreover, we would use boldface letters to denote vectors in $\Fp^2$, for example $\bh,\bn,\bm$ would mean $(h_1,h_2),(n_1,n_2),(m_1,m_2)$ respectively.

\subsection{Fourier Analysis Notation}
Let $d$ be a positive integer, for any function $f:\mathbb{F}_p^d\longrightarrow \mathbb{C}$, denote the $L^r$-norm with respect to the uniform probability measure on $\mathbb{F}_p^d$ by $\lVert f \rVert_r=\bigl(\frac{1}{p^d}\sum_{x}\lvert f(x)\rvert^r\bigr)^{1/r}$ and the $l^r$-norm with respect to the counting measure on $\mathbb{F}_p^d$ by $\lVert f \rVert_{l^r}=\bigl(\sum_{x}\lvert f(x)\rvert^r\bigr)^{1/r}$. Denote $e(x)=e^{2\pi i x}$, $e_p(x)=e^{2\pi i x/p}$ and Fourier transform $\widehat{f}(\xi)=\mathbb{E}_{x} f(x)e_p(-x\xi)$. In this convention, Parseval's identity reads $\lVert f \rVert_2=\lVert \widehat{f} \rVert_{l^2}$ and Fourier inversion reads $f(x)=\sum_{\xi}\widehat{f}(\xi)e_p(\xi x)$.

\subsection{Gowers Norm Notation}

Let $G$ be a finite abelian group. Denote the difference operator by $\Delta$, that is $\Delta_hg(x)=g(x)\overline{g}(x+h)$ for any functions $g$ on $G$ and any elements $h\in G$. Let $s$ be a positive integer and $f:G\longrightarrow\C$ be a function, the $U^s(G)$-Gowers norm is defined to be
\begin{equation*}
    \lVert f \rVert_{U^s(G)}^{2^s}=\E_{x\in G}\;\E_{h_1,h_2,...,h_s\in G} \Delta_{h_1}\Delta_{h_2}...\Delta_{h_s}f(x).
\end{equation*}
More generally, let $H_1,H_2,...,H_s$ be subgroups of $G$, the Gowers box norm with respect to $H_1,H_2,...,H_s$ is defined to be
\begin{equation*}
    \lVert f \rVert_{H_1,H_2,...,H_s}^{2^s}=\E_{x\in G}\;\E_{h_1\in H_1,h_2\in H_2,...,h_s\in H_s} \Delta_{h_1}\Delta_{h_2}...\Delta_{h_s}f(x).
\end{equation*}
When $H=H_1=H_2=...=H_s$, we would write $\lVert f \rVert_{U^s(H)}$ for $\lVert f \rVert_{H_1,H_2,...,H_s}$.

Suppose $K_1,K_2,...,K_s$ are subgroups of $H_1,H_2,...,H_s$ respectively, then we have
\begin{equation*}
    \lVert f \rVert_{H_1,H_2,...,H_s}\leq\lVert f \rVert_{K_1,K_2,...,K_s}
\end{equation*}
for all functions $f$.

\subsection{Algebraic Geometry Terminology} By (classical) varieties, we mean any classical affine, quasi-affine, projective or quasi-projective varieties. We do not require the varieties to be irreducible, hence our affine varieties and projective varieties may be algebraic sets in some references. By abuse of terminology, we also refer to schemes that are separated and of finite type over $\mathrm{Spec}\;k$ for some field $k$ as a variety. We would often use $Y(\mathbb{C}),Z(\mathbb{F}_p),...$ to denote classical varieties and $Y,Z,...$ to denote schemes. If $X$ is a scheme over some field $k$, we also write $X(k)$ as $k$-valued points $\mathrm{Hom}_{k}(\mathrm{Spec}\;k,X)$. 

If $X$ is a variety or scheme, denote the dimension of $X$ as a variety (or scheme) by $\dim_{\mathrm{Var}}X$ (or $\dim_{\mathrm{Sch}}X$), which is the supremum of the length of the chains of irreducible closed subsets in the underlying topological space.

\section{Algebraic geometry PET induction}\label{sec: PET}

In this section, we would establish the framework of algebraic geometry PET induction. The main purpose is to upper bound the counting operator for corners generated by rational functions $P(t),Q(t)$ by a Gowers norm, modulo the point-counting estimates for certain varieties. The precise statement is in Theorem \ref{thm: PET}, and the corresponding point-counting estimates would be handled in Section \ref{sec: dim}.

\begin{defn}
    Let $P(t),Q(t)\in\Q(t)$ be rational functions over $\Q$. Define the \textbf{Roth variety} for corners associated to $P(t),Q(t)$ to be the (quasi-affine) variety cut out by the following ten equations in sixteen variables $y_1,y_2,...,y_{16}$.
    \begin{align}
        P(y_1)-P(y_3)&-P(y_9)+P(y_{11})=0\label{eqs: roth1}\\
        P(y_2)-P(y_4)&-P(y_{10})+P(y_{12})=0\label{eqs: roth2}\\
        P(y_5)-P(y_7)&-P(y_{13})+P(y_{15})=0\label{eqs: roth3}\\
        P(y_6)-P(y_8)&-P(y_{14})+P(y_{16})=0\label{eqs: roth4}\\
        Q(y_1)-Q(y_2)&-Q(y_{9})+Q(y_{10})=0\label{eqs: roth5}\\
        Q(y_3)-Q(y_7)&-Q(y_{11})+Q(y_{15})=0\label{eqs: roth6}\\
        Q(y_4)-Q(y_8)&-Q(y_{12})+Q(y_{16})=0\label{eqs: roth7}\\
        Q(y_5)-Q(y_6)&-Q(y_{13})+Q(y_{14})=0\label{eqs: roth8}\\
        P(y_{9})-P(y_{10})-P(y_{11})+P(y_{12})&-P(y_{13})+P(y_{14})+P(y_{15})-P(y_{16})=0\label{eqs: roth9}\\
        Q(y_{9})-Q(y_{10})-Q(y_{11})+Q(y_{12})&-Q(y_{13})+Q(y_{14})+Q(y_{15})-Q(y_{16})=0\label{eqs: roth10}
    \end{align}
    The Roth variety could be regarded as a classical variety or a scheme over various fields, for example, the complex numbers $\C$ or finite fields $\Fp$.
\end{defn}

\begin{thm}\label{thm: PET}
    Let $P(t),Q(t)\in \Q(t)$ be rational functions over $\Q$, and $Y(\Fp)$ be the $\Fp$-points of the Roth variety for corners associated to $P(t),Q(t)$. Then, we have the following Gowers norm control
    \begin{multline*}
        \biggl\lvert\E_{x_1,x_2,y} f_0(x_1,x_2) f_1(x_1+P(y),x_2) f_2(x_1,x_2+Q(y))\biggr\rvert\\
        \leq \lVert f_0 \rVert_2 \lVert f_1 \rVert_4 \lVert f_2 \rVert_4^{\frac{1}{2}} \lVert f_2 \rVert_{U^2(0\times \Fp)}^{\frac{1}{4}} \biggl(\frac{1}{p^6}|Y(\Fp)|\biggr)^{\frac{1}{16}}
    \end{multline*}
    for all $1$-bounded functions $f_0,f_1,f_2:\Fp^2\longrightarrow \C$.
\end{thm}

Heuristically, the Roth variety for corners is cut out by $10$ constraints in $16$ variables, hence the intuitive expectation suggests that the dimension of the Roth variety is $16-10=6$ and the size of $\Fp$-points on the Roth variety is about $p^{6}$ when the $10$ constraints are ``sufficiently independent". In this case, Theorem \ref{thm: PET} would indeed be providing control by the Gowers norm.

The ingredients of the proof of Theorem \ref{thm: PET} are divided into three propositions, namely Propositions \ref{prop: PET}, \ref{prop: Fourier1} and \ref{prop: Fourier2}. The strategy is to generalize the algebraic geometry PET induction for one dimensional three term rational function progressions developed by Hong and the author \cite{hong-lim} to two dimensional corners. We would proceed with Fourier analytic method, and transform the exponential sums into point-counting estimates for the corresponding Roth variety to avoid difficult multidimensional exponential sum estimates. The Fourier calculations in the first half (up until Equation \eqref{eqs: PET7}) of the proof of Proposition \ref{prop: PET} are same as the Fourier calculations in the work of Han, Lacey and Yang \cite{han-lacey-yang} on corners generated by polynomials. It seems plausible that combining the ideas from \cite{hong-lim} and \cite{han-lacey-yang} would lead to the desired estimates, however, there is a new obstacle--the ``naive" generalization of algebraic geometry PET induction produces varieties with wrong expected dimensions. Fortunately, this obstacle could be tackled by ideas in the work of Kavrut and Wu \cite{kavrut-wu} on improving the nonlinear Roth theorem.

The Proposition \ref{prop: PET} is an intermediate step for algebraic geometry PET induction; it has the correct geometric term $|Y(\Fp)|$ but messy Fourier terms $F_1,F_2$. The Fourier terms would be simplified in Proposition \ref{prop: Fourier1} and \ref{prop: Fourier2}. In the statement and proof of Proposition \ref{prop: PET}, we would frequently use boldface letters to denote vectors in $\Fp^2$, for example $\bh,\bn,\bm$ would mean $(h_1,h_2),(n_1,n_2),(m_1,m_2)$ respectively. Also, recall that the difference operator $\Delta$ is defined as $\Delta_{\bh} g(\bx)=g(\bx)\overline{g}(\bx+\bh)$ for any functions $g$. 

\begin{prop}\label{prop: PET}
    Let $P(t),Q(t)\in \Q(t)$ be rational functions over $\Q$, and $Y(\Fp)$ be the $\Fp$-points of the Roth variety for corners associated to $P(t),Q(t)$. For any functions $f_1,f_2$, define
    \begin{align*}
        F_1(n_1,\bh)&\coloneqq\sum_{n_2}\Delta_{-\bh}\widehat{f}_1(n_1,n_2)\\
        F_2(m_2,\bh)&\coloneqq\sum_{m_1}\Delta_{\bh}\widehat{f}_2(m_1,m_2),
    \end{align*}
    where the order of operators is taking Fourier transform first and then difference operator for $\Delta_{\bh}\widehat{f}$. Then, we have
    \begin{multline*}
        \biggl\lvert\E_{x_1,x_2,y} f_0(x_1,x_2) f_1(x_1+P(y),x_2) f_2(x_1,x_2+Q(y))\biggr\rvert\\
        \leq\lVert f_0 \rVert_2 \lVert F_1(n_1,\bh) \rVert_{l^2_{n_1,\bh}}^{1/2}\lVert F_2(m_2,\bh) \rVert_{l^2_{m_2,\bh}}^{1/4}\lVert \overline{F_2}(m_2^{\prime},\bh)F_2(m_2^{\dprime},\bh) \rVert_{l^2_{m_2^{\prime},m_2^{\dprime},\bh}}^{1/8}\biggl(\frac{1}{p^6}|Y(\Fp)|\biggr)^{1/16}
    \end{multline*}
    for all functions $f_0,f_1,f_2:\Fp^2\longrightarrow \C$.
\end{prop}

\begin{proof}
    For notational convenience, denote 
    \begin{equation*}
        \Lambda_0 \coloneqq \biggl\lvert\E_{x_1,x_2,y} f_0(x_1,x_2) f_1(x_1+P(y),x_2) f_2(x_1,x_2+Q(y))\biggr\rvert
    \end{equation*}
    to be the quantity we would like to estimate.
    
    Applying Cauchy--Schwarz inequality in $x_1,x_2$, we have
    \begin{equation*}
        \begin{split}
            \Lambda_0 &=\biggl\lvert\E_{x_1,x_2} f_0(x_1,x_2) \E_{y}f_1(x_1+P(y),x_2) f_2(x_1,x_2+Q(y))\biggr\rvert\\
            &\leq \lVert f_0 \rVert_2 \Biggl(\E_{x_1,x_2}\biggl\lvert\E_{y}f_1(x_1+P(y),x_2) f_2(x_1,x_2+Q(y))\biggr\rvert^2\Biggr)^{1/2}.
        \end{split}
    \end{equation*}
    
    Fourier expand the functions $f_1,f_2$ to obtain
    \begin{equation}\label{eqs: PET1}
        \begin{split}
            &\E_{y}f_1(x_1+P(y),x_2)f_2(x_1,x_2+Q(y))\\
            &=\E_{y}\sum_{\bn,\bm} \widehat{f}_1(\bn)e_p(\bn\cdot\bx)e_p(n_1P(y))\widehat{f}_2(\bm)e_p(\bm\cdot\bx)e_p(m_2Q(y))\\
            &=\sum_{\bn,\bm} \widehat{f}_1(\bn)\widehat{f}_2(\bm)e_p((\bn+\bm)\cdot\bx)\E_{y}e_p(n_1P(y)+m_2Q(y)).
        \end{split}
    \end{equation}

    Denote $K(a,b)=\E_{y}e_p(aP(y)+bQ(y))$ and perform change of variables $\bn\leftrightarrow\bn-\bm$, then the expression \eqref{eqs: PET1} is equal to
    \begin{equation}\label{eqs: PET2}
        \begin{split}
            &\sum_{\bn,\bm} \widehat{f}_1(\bn)\widehat{f}_2(\bm)e_p((\bn+\bm)\cdot\bx)K(n_1,m_2)\\
            &=\sum_{\bn,\bm} \widehat{f}_1(\bn-\bm)\widehat{f}_2(\bm)e_p(\bn\cdot\bx)K(n_1-m_1,m_2)\\
            &=\sum_{\bn}\biggl(\sum_{\bm} \widehat{f}_1(\bn-\bm)\widehat{f}_2(\bm)K(n_1-m_1,m_2)\biggr)e_p(\bn\cdot\bx)
        \end{split}
    \end{equation}

    Combining \eqref{eqs: PET1} and \eqref{eqs: PET2}, Parseval's identity says
    \begin{equation}\label{eqs: PET3}
        \begin{split}            &\E_{x_1,x_2}\biggl\lvert\E_{y}f_1(x_1+P(y),x_2) f_2(x_1,x_2+Q(y))\biggr\rvert^2\\
        &=\sum_{\bn}\Biggl\lvert\sum_{\bm} \widehat{f}_1(\bn-\bm)\widehat{f}_2(\bm)K(n_1-m_1,m_2)\Biggr\rvert^2.
        \end{split}
    \end{equation}

    Expanding the squares and performing change of variables $\bn\leftrightarrow\bn+\bm$, the expression \eqref{eqs: PET3} becomes
    \begin{equation}\label{eqs: PET4}
        \begin{split}
            &\sum_{\bn}\sum_{\bm,\bh} \widehat{f}_1(\bn-\bm)\overline{\widehat{f}_1}(\bn-\bm-\bh)\widehat{f}_2(\bm)\overline{\widehat{f}_2}(\bm+\bh)K(n_1-m_1,m_2)\overline{K}(n_1-m_1-h_1,m_2+h_2)\\
            &=\sum_{\bn,\bm,\bh} \widehat{f}_1(\bn)\overline{\widehat{f}_1}(\bn-\bh)\widehat{f}_2(\bm)\overline{\widehat{f}_2}(\bm+\bh)K(n_1,m_2)\overline{K}(n_1-h_1,m_2+h_2)\\
            &=\sum_{\bn,\bm,\bh} \Delta_{-\bh}\widehat{f}_1(\bn)\Delta_{\bh}\widehat{f}_2(\bm)\Delta_{(-h_1,h_2)}K(n_1,m_2).
        \end{split}
    \end{equation}
    To clarify, for the notation $\Delta_{\bh}\widehat{f}$, the order of operators is applying Fourier transform first and then difference operator. 

    Note that in the last line of \eqref{eqs: PET4}, some variables decouple, hence we can write \eqref{eqs: PET4} as
    \begin{equation}\label{eqs: PET5}
        \sum_{\bh,n_1,m_2} \Delta_{(-h_1,h_2)}K(n_1,m_2) \Biggl(\sum_{n_2}\Delta_{-\bh}\widehat{f}_1(n_1,n_2)\Biggr) \Biggl(\sum_{m_1}\Delta_{\bh}\widehat{f}_2(m_1,m_2)\Biggr).
    \end{equation}

    To ease the notations, define
    \begin{align*}
        F_1(n_1,\bh)&\coloneqq\sum_{n_2}\Delta_{-\bh}\widehat{f}_1(n_1,n_2)\\
        F_2(m_2,\bh)&\coloneqq\sum_{m_1}\Delta_{\bh}\widehat{f}_2(m_1,m_2),
    \end{align*}
    so that the expression \eqref{eqs: PET5} can be rewritten more succinctly as 
    \begin{equation}\label{eqs: PET6}
        \sum_{\bh,n_1,m_2} F_1(n_1,\bh)F_2(m_2,\bh)\Delta_{(-h_1,h_2)}K(n_1,m_2). 
    \end{equation}

    Applying Cauchy--Schwarz inequality in variables $\bh,n_1$, we deduce that
    \begin{equation*}
        \begin{split}
            \eqref{eqs: PET6}&=\sum_{\bh,n_1} F_1(n_1,\bh)\sum_{m_2}F_2(m_2,\bh)\Delta_{(-h_1,h_2)}K(n_1,m_2)\\
            &\leq \lVert F_1(n_1,\bh) \rVert_{l^2_{n_1,\bh}} \Biggl(\sum_{\bh,n_1}\sum_{m_2,m_2^{\prime}} F_2(m_2,\bh)\overline{F_2}(m_2^{\prime},\bh)\Delta_{(-h_1,h_2)}K(n_1,m_2)\overline{\Delta_{(-h_1,h_2)}K}(n_1,m_2^{\prime}) \Bigg)^{1/2}.
        \end{split}
    \end{equation*}

    For notational convenience, denote
    \begin{equation*}
        \Lambda_1\coloneqq \sum_{\bh,n_1}\sum_{m_2,m_2^{\prime}} F_2(m_2,\bh)\overline{F_2}(m_2^{\prime},\bh)\Delta_{(-h_1,h_2)}K(n_1,m_2)\overline{\Delta_{(-h_1,h_2)}K}(n_1,m_2^{\prime}).
    \end{equation*}
    
    To sum up, so far we have obtained
    \begin{equation}\label{eqs: PET7}
        \Lambda_0\leq \lVert f_0 \rVert_2 \lVert F_1(n_1,\bh) \rVert_{l^2_{n_1,\bh}}^{1/2} \Lambda_1^{1/4}
    \end{equation}
    and we have Cauchy--Schwarz away the functions $f_0,f_1$. It is tempting to apply Cauchy--Schwarz inequality to $\Lambda_1$ in variables $\bh,m_2,m_2^{\prime}$ to get rid of $f_2$. However, in this way, we would discard too much cancellation and the resulting variety has degeneracy, that is, the variety is defined by six equations in eight variables, but its dimension is three. To fix this, we use a technique due to Kavrut and Wu \cite{kavrut-wu} to iterate the exponential sums one more time in order to capture more cancellation. Hence, we apply Cauchy--Schwarz inequality to $\Lambda_1$ in variables $\bh,m_2$ to obtain
    \begin{equation}\label{eqs: PET8}
        \begin{split}
            \Lambda_1&=\sum_{\bh,m_2}F_2(m_2,\bh)\sum_{n_1,m_2^{\prime}} \overline{F_2}(m_2^{\prime},\bh)\Delta_{(-h_1,h_2)}K(n_1,m_2)\overline{\Delta_{(-h_1,h_2)}K}(n_1,m_2^{\prime})\\
            &\leq \lVert F_2(m_2,\bh) \rVert_{l^2_{m_2,\bh}}\\
            &\quad\times\vast(\sum_{\bh,m_2}\sum_{\substack{n_1,m_2^{\prime}\\n_1^{\prime},m_2^{\dprime}}} \overline{F_2}(m_2^{\prime},\bh)F_2(m_2^{\dprime},\bh)\Delta_{(-h_1,h_2)}K(n_1,m_2)\overline{\Delta_{(-h_1,h_2)}K}(n_1^{\prime},m_2)\\
            &\quad\quad\quad\quad\quad\quad\quad\quad\quad\quad\quad\quad\quad\quad\quad\quad
            \overline{\Delta_{(-h_1,h_2)}K}(n_1,m_2^{\prime})\Delta_{(-h_1,h_2)}K(n_1^{\prime},m_2^{\dprime})\vast)^{1/2},
        \end{split}
    \end{equation}
    where $m_2^{\dprime},n_1^{\prime}$ are the copies of $m_2^{\prime},n_1$ respectively.

    Once again, for notational convenience, denote
    \begin{equation*}
        \begin{split}
            \Lambda_2&\coloneqq\sum_{\bh,m_2}\sum_{\substack{n_1,m_2^{\prime}\\n_1^{\prime},m_2^{\dprime}}} \overline{F_2}(m_2^{\prime},\bh)F_2(m_2^{\dprime},\bh)\Delta_{(-h_1,h_2)}K(n_1,m_2)\overline{\Delta_{(-h_1,h_2)}K}(n_1^{\prime},m_2)\\
            &\quad\quad\quad\quad\quad\quad\quad\quad\quad\quad\quad\quad\quad\quad
            \overline{\Delta_{(-h_1,h_2)}K}(n_1,m_2^{\prime})\Delta_{(-h_1,h_2)}K(n_1^{\prime},m_2^{\dprime})
        \end{split}
    \end{equation*}
    and \eqref{eqs: PET8} can be rewritten as
    \begin{equation}\label{eqs: PET9}
        \Lambda_1\leq\lVert F_2(m_2,\bh) \rVert_{l^2_{m_2,\bh}}\Lambda_2^{1/2}.
    \end{equation}

    Now, we have kept enough cancellation, and we can apply Cauchy--Schwarz inequality to $\Lambda_2$ in variables $\bh,m_2^{\prime},m_2^{\dprime}$ to get rid of $F_2$ and transform exponential sums into point-counting on variety. Hence, we have
    \begin{equation}\label{eqs: PET10}
        \begin{split}
            \Lambda_2&=\sum_{\bh,m_2^{\prime},m_2^{\dprime}}\overline{F_2}(m_2^{\prime},\bh)F_2(m_2^{\dprime},\bh)\sum_{n_1,n_1^{\prime},m_2}\Delta_{(-h_1,h_2)}K(n_1,m_2)\overline{\Delta_{(-h_1,h_2)}K}(n_1^{\prime},m_2)\\
            &\quad\quad\quad\quad\quad\quad\quad\quad\quad\quad\quad\quad\quad\quad\quad\quad\quad
            \overline{\Delta_{(-h_1,h_2)}K}(n_1,m_2^{\prime})\Delta_{(-h_1,h_2)}K(n_1^{\prime},m_2^{\dprime})\\
            &\leq \lVert \overline{F_2}(m_2^{\prime},\bh)F_2(m_2^{\dprime},\bh) \rVert_{l^2_{m_2^{\prime},m_2^{\dprime},\bh}}\\
            &\quad\times\vast(\sum_{\bh,m_2^{\prime},m_2^{\dprime}}\sum_{\substack{n_1,n_1^{\prime},m_2\\N_1,N_1^{\prime},M_2}}\Delta_{(-h_1,h_2)}K(n_1,m_2)\overline{\Delta_{(-h_1,h_2)}K}(n_1^{\prime},m_2)\overline{\Delta_{(-h_1,h_2)}K}(n_1,m_2^{\prime})\Delta_{(-h_1,h_2)}K(n_1^{\prime},m_2^{\dprime})\\
            &\quad\quad\quad\quad\quad\quad\quad 
            \overline{\Delta_{(-h_1,h_2)}K}(N_1,M_2)\Delta_{(-h_1,h_2)}K(N_1^{\prime},M_2)\Delta_{(-h_1,h_2)}K(N_1,m_2^{\prime})\overline{\Delta_{(-h_1,h_2)}K}(N_1^{\prime},m_2^{\dprime})\vast)^{1/2},
        \end{split}
    \end{equation}
    where $N_1,N_1^{\prime},M_2$ are the copies of $n_1,n_1^{\prime},m_2$ respectively.

    Finally, expanding the exponential sums, we have
    \begin{equation}\label{eqs: PET11}
        \begin{split}
            &\sum_{\bh,m_2^{\prime},m_2^{\dprime}}\sum_{\substack{n_1,n_1^{\prime},m_2\\N_1,N_1^{\prime},M_2}}\Delta_{(-h_1,h_2)}K(n_1,m_2)\overline{\Delta_{(-h_1,h_2)}K}(n_1^{\prime},m_2)\overline{\Delta_{(-h_1,h_2)}K}(n_1,m_2^{\prime})\Delta_{(-h_1,h_2)}K(n_1^{\prime},m_2^{\dprime})\\
            &\quad\quad\quad\quad\quad\quad\quad 
            \overline{\Delta_{(-h_1,h_2)}K}(N_1,M_2)\Delta_{(-h_1,h_2)}K(N_1^{\prime},M_2)\Delta_{(-h_1,h_2)}K(N_1,m_2^{\prime})\overline{\Delta_{(-h_1,h_2)}K}(N_1^{\prime},m_2^{\dprime})\\
            &=\E_{\substack{y_1,y_2,y_3,y_4\\y_5,y_6,y_7,y_8\\y_9,y_{10},y_{11},y_{12}\\y_{13},y_{14},y_{15},y_{16}}}\sum_{\substack{h_1,h_2\\n_1,n_1^{\prime},N_1,N_1^{\prime}\\m_2,m_2^{\prime},m_2^{\dprime},M_2}}\\
            &\quad\quad\quad\quad\quad\quad\quad
            e_p\Bigl(+n_1P(y_1)+m_2Q(y_1) \quad -(n_1-h_1)P(y_9)-(m_2+h_2)Q(y_9)\\
            &\quad\quad\quad\quad\quad\quad\quad\quad\,
            -n_1^{\prime}P(y_2)-m_2Q(y_2) \quad +(n_1^{\prime}-h_1)P(y_{10})+(m_2+h_2)Q(y_{10})\\
            &\quad\quad\quad\quad\quad\quad\quad\quad\,
            -n_1P(y_3)-m_2^{\prime}Q(y_3) \quad +(n_1-h_1)P(y_{11})+(m_2^{\prime}+h_2)Q(y_{11})\\
            &\quad\quad\quad\quad\quad\quad\quad\quad\,
            +n_1^{\prime}P(y_4)+m_2^{\dprime}Q(y_4) \quad -(n_1^{\prime}-h_1)P(y_{12})-(m_2^{\dprime}+h_2)Q(y_{12})\\
            &\quad\quad\quad\quad\quad\quad\quad\quad\,
            -N_1P(y_5)-M_2Q(y_5) \quad +(N_1-h_1)P(y_{13})+(M_2+h_2)Q(y_{13})\\
            &\quad\quad\quad\quad\quad\quad\quad\quad\,
            +N_1^{\prime}P(y_6)+M_2Q(y_6) \quad -(N_1^{\prime}-h_1)P(y_{14})-(M_2+h_2)Q(y_{14})\\
            &\quad\quad\quad\quad\quad\quad\quad\quad\,
            +N_1P(y_7)+m_2^{\prime}Q(y_7) \quad -(N_1-h_1)P(y_{15})-(m_2^{\prime}+h_2)Q(y_{15})\\
            &\quad\quad\quad\quad\quad\quad\quad\quad\,
            -N_1^{\prime}P(y_8)-m_2^{\dprime}Q(y_8) \quad +(N_1^{\prime}-h_1)P(y_{16})+(m_2^{\dprime}+h_2)Q(y_{16})\Bigr)
        \end{split}
    \end{equation}
    which is equal to
    \begin{equation}\label{eqs: PET12}
        \begin{split}
            &=\E_{\substack{y_1,y_2,y_3,y_4\\y_5,y_6,y_7,y_8\\y_9,y_{10},y_{11},y_{12}\\y_{13},y_{14},y_{15},y_{16}}}\sum_{\substack{h_1,h_2\\n_1,n_1^{\prime},N_1,N_1^{\prime}\\m_2,m_2^{\prime},m_2^{\dprime},M_2}}\\
            &\quad\quad\quad\quad\quad\quad\quad
            e_p\Bigl(n_1(P(y_1)-P(y_3)-P(y_9)+P(y_{11}))\\
            &\quad\quad\quad\quad\quad\quad\quad\quad\,\,\,
            +n_1^{\prime}(-P(y_2)+P(y_4)+P(y_{10})-P(y_{12}))\\
            &\quad\quad\quad\quad\quad\quad\quad\quad\,\,\,
            +N_1(-P(y_5)+P(y_7)+P(y_{13})-P(y_{15}))\\
            &\quad\quad\quad\quad\quad\quad\quad\quad\,\,\,
            +N_1^{\prime}(P(y_6)-P(y_8)-P(y_{14})+P(y_{16}))\\
            &\quad\quad\quad\quad\quad\quad\quad\quad\,\,\,
            +m_2(Q(y_1)-Q(y_2)-Q(y_{9})+Q(y_{10}))\\
            &\quad\quad\quad\quad\quad\quad\quad\quad\,\,\,
            +m_2^{\prime}(-Q(y_3)+Q(y_7)+Q(y_{11})-Q(y_{15}))\\
            &\quad\quad\quad\quad\quad\quad\quad\quad\,\,\,
            +m_2^{\dprime}(Q(y_4)-Q(y_8)-Q(y_{12})+Q(y_{16}))\\
            &\quad\quad\quad\quad\quad\quad\quad\quad\,\,\,
            +M_2(-Q(y_5)+Q(y_6)+Q(y_{13})-Q(y_{14}))\\
            &\quad\quad\quad\quad\quad\quad\quad\quad\,\,\,
            +h_1(P(y_{9})-P(y_{10})-P(y_{11})+P(y_{12})-P(y_{13})+P(y_{14})+P(y_{15})-P(y_{16}))\\
            &\quad\quad\quad\quad\quad\quad\quad\quad\,\,\,
            +h_2(-Q(y_{9})+Q(y_{10})+Q(y_{11})-Q(y_{12})+Q(y_{13})-Q(y_{14})-Q(y_{15})+Q(y_{16}))\Bigr)\\
            &=\frac{1}{p^6}|Y(\Fp)|
        \end{split}
    \end{equation}
    where the last line of \eqref{eqs: PET12} follows from the orthogonality of characters and the definition of the Roth variety.

    To complete the proof, we combine \eqref{eqs: PET7}, \eqref{eqs: PET9}, \eqref{eqs: PET10}, \eqref{eqs: PET11} and \eqref{eqs: PET12}, and conclude
    \begin{align*}        
        \Lambda_0&\leq \lVert f_0 \rVert_2 \lVert F_1(n_1,\bh) \rVert_{l^2_{n_1,\bh}}^{1/2} \Lambda_1^{1/4}\\
        &\leq\lVert f_0 \rVert_2 \lVert F_1(n_1,\bh) \rVert_{l^2_{n_1,\bh}}^{1/2}\lVert F_2(m_2,\bh) \rVert_{l^2_{m_2,\bh}}^{1/4}\Lambda_2^{1/8}.\\
        &\leq\lVert f_0 \rVert_2 \lVert F_1(n_1,\bh) \rVert_{l^2_{n_1,\bh}}^{1/2}\lVert F_2(m_2,\bh) \rVert_{l^2_{m_2,\bh}}^{1/4}\lVert \overline{F_2}(m_2^{\prime},\bh)F_2(m_2^{\dprime},\bh) \rVert_{l^2_{m_2^{\prime},m_2^{\dprime},\bh}}^{1/8}\biggl(\frac{1}{p^6}|Y(\Fp)|\biggr)^{1/16}.
    \end{align*}
\end{proof}

\begin{rem}
    Note that we do not consider summing over $n_1$ for $\Lambda_1$ to exploit the cancellation among exponential sums since we want to avoid difficult multidimensional exponential sums with rational function phase functions.
\end{rem}

The next proposition is the first part of simplifications of Fourier coefficients in Proposition \ref{prop: PET}. Similar calculations also appear in \cite{han-lacey-yang}.

\begin{prop}\label{prop: Fourier1}
    Let $f_1,f_2:\Fp^2\longrightarrow\C$ be functions, define 
    \begin{equation*}
        F_1(n_1,\bh)=\sum_{n_2}\Delta_{-\bh}\widehat{f}_1(n_1,n_2),\quad F_2(m_2,\bh)=\sum_{m_1}\Delta_{\bh}\widehat{f}_2(m_1,m_2),
    \end{equation*}
    then we have
    \begin{equation*}
        \lVert F_1(n_1,\bh) \rVert_{l^2_{n_1,\bh}}\leq \lVert f_1 \rVert_4^2,\quad \lVert F_2(m_2,\bh) \rVert_{l^2_{m_2,\bh}}\leq \lVert f_2 \rVert_4^2.
    \end{equation*}
\end{prop}

\begin{proof}
    We would prove the estimate for $F_2$ and the estimate for $F_1$ follows similarly by symmetry. By definition of $F_2$ and Fourier transform, we have
    \begin{equation}\label{eqs: F1_1}
        \begin{split}
            F_2(m_2,\bh)&=\sum_{m_1}\Delta_{\bh}\widehat{f}_2(m_1,m_2)\\
            &=\sum_{m_1}\widehat{f}_2(m_1,m_2)\overline{\widehat{f}_2}(m_1+h_1,m_2+h_2)\\
            &=\sum_{m_1}\E_{\bx,\bz}f_2(\bx)e_p(-\bm\cdot\bx)\overline{f_2}(\bz)e_p(\bm\cdot\bz)e_p(\bh\cdot\bz).
        \end{split}
    \end{equation}

    Summing over variable $m_1$ to utilize the orthogonality of characters, the expression \eqref{eqs: F1_1} becomes
    \begin{equation}\label{eqs: F1_2}
        \begin{split}
            &\E_{x_2,z_1,z_2}f_2(z_1,x_2)\overline{f_2}(z_1,z_2)e_p(m_2(z_2-x_2))e_p(\bh\cdot\bz)\\
        &=\E_{x_2,z_1,z_2}f_2(z_1,x_2+z_2)\overline{f_2}(z_1,z_2)e_p(\bh\cdot\bz-m_2x_2)
        \end{split}
    \end{equation}
    where the last line of \eqref{eqs: F1_2} follows from change of variables $x_2\leftrightarrow x_2+z_2$.

    Combining \eqref{eqs: F1_1} and \eqref{eqs: F1_2}, Parseval's identity says
    \begin{equation}\label{eqs: F1_3}
        \lVert F_2(m_2,\bh) \rVert_{l^2_{m_2,\bh}}^2=\sum_{m_2,\bh} \bigl\lvert F_2(m_2,\bh) \bigr\rvert^2=\E_{x_2,z_1,z_2}\bigl\lvert f_2(z_1,x_2+z_2)\overline{f_2}(z_1,z_2)\bigr\rvert^2.
    \end{equation}

    Note that
    \begin{equation}\label{eqs: F1_4}
        \begin{split}
            \E_{x_2,z_1,z_2}\bigl\lvert f_2(z_1,x_2+z_2)\overline{f_2}(z_1,z_2)\bigr\rvert^2&=\E_{x_2,z_1,z_2}\bigl\lvert f_2(z_1,x_2)f_2(z_1,z_2)\bigr\rvert^2\\
            &=\E_{z_1}\biggl(\E_{x_2}\bigl\lvert f_2(z_1,x_2)\bigr\rvert^2\biggr)\biggl(\E_{z_2}\bigl\lvert f_2(z_1,z_2)\bigr\rvert^2\biggr)\\
            &=\E_{z_1}\biggl(\E_{z_2}\bigl\lvert f_2(z_1,z_2)\bigr\rvert^2\biggr)^2
        \end{split}
    \end{equation}

    Finally, applying Cauchy--Schwarz inequality to the inner sum of the last line of \eqref{eqs: F1_4}, we have
    \begin{equation}\label{eqs: F1_5}
        \E_{z_1}\biggl(\E_{z_2}\bigl\lvert f_2(z_1,z_2)\bigr\rvert^2\biggr)^2\leq \E_{z_1}\E_{z_2}\bigl\lvert f_2(z_1,z_2)\bigr\rvert^4=\lVert f_2 \rVert_4^4.
    \end{equation}
    Combining \eqref{eqs: F1_3}, \eqref{eqs: F1_4} and \eqref{eqs: F1_5} would complete the proof.
\end{proof}

The next proposition is the second part of simplifications of Fourier coefficients in Proposition \ref{prop: PET}. Compared to the Proposition \ref{prop: Fourier1}, its proof is more complicated and involves Gowers norms.

\begin{prop}\label{prop: Fourier2}
    Let $f_2:\Fp^2\longrightarrow\C$ be a $1$-bounded function and define $F_2(m_2,\bh)=\sum_{m_1}\Delta_{\bh}\widehat{f}_2(m_1,m_2)$. Then, we have
    \begin{equation*}
        \lVert \overline{F_2}(m_2,\bh)F_2(m_2^{\prime},\bh) \rVert_{l^2_{m_2,m_2^{\prime},\bh}}\leq \lVert f_2 \rVert_{U^2(0\times\Fp)}^2.
    \end{equation*}
\end{prop}

\begin{proof}
    By definition of $F_2$ and Fourier transform, we have
    \begin{equation}\label{eqs: F2_1}
        \begin{split}
            F_2(m_2,\bh)&=\sum_{m_1}\Delta_{\bh}\widehat{f}_2(m_1,m_2)\\
            &=\sum_{m_1}\widehat{f}_2(m_1,m_2)\overline{\widehat{f}_2}(m_1+h_1,m_2+h_2)\\
            &=\sum_{m_1}\E_{\bx,\bz}f_2(\bx)e_p(-\bm\cdot\bx)\overline{f_2}(\bz)e_p(\bm\cdot\bz)e_p(\bh\cdot\bz).
        \end{split}
    \end{equation}

    Summing over variable $m_1$ to utilize the orthogonality of characters, the expression \eqref{eqs: F2_1} becomes
    \begin{equation}\label{eqs: F2_2}
        \begin{split}
            &\E_{x_2,z_1,z_2}f_2(z_1,x_2)\overline{f_2}(z_1,z_2)e_p(m_2(z_2-x_2))e_p(\bh\cdot\bz)\\
        &=\E_{x_2}\biggl(\E_{z_1,z_2}f_2(z_1,x_2+z_2)\overline{f_2}(z_1,z_2)e_p(\bh\cdot\bz)\biggr)e_p(-m_2x_2)
        \end{split}
    \end{equation}
    where the last line of \eqref{eqs: F2_2} follows from change of variables $x_2\leftrightarrow x_2+z_2$.

    Combining \eqref{eqs: F2_1} and \eqref{eqs: F2_2}, Parseval's identity says
    \begin{equation}\label{eqs: F2_3}
        \begin{split}
            \sum_{m_2}\bigl\lvert F_2(m_2,\bh)\bigr\rvert^2
            &=\E_{x_2}\biggl\lvert\E_{\bz}f_2(z_1,x_2+z_2)\overline{f_2}(z_1,z_2)e_p(\bh\cdot\bz)\biggr\rvert^2\\
            &=\E_{x_2}\E_{\bz,\bz^{\prime}}\Delta_{\bz^{\prime}}f_2(z_1,x_2+z_2)\overline{\Delta_{\bz^{\prime}}f_2}(z_1,z_2)e_p(-\bh\cdot\bz^{\prime}).
        \end{split}
    \end{equation}

    Note that the quantity we would like to upper bound for this proof has the following decomposition:
    \begin{equation}\label{eqs: F2_4}
        \begin{split}
            \lVert \overline{F_2}(m_2,\bh)F_2(m_2^{\prime},\bh) \rVert_{l^2_{m_2,m_2^{\prime},\bh}}^2&=\sum_{m_2,m_2^{\prime},\bh}\bigl\lvert F_2(m_2,\bh)F_2(m_2^{\prime},\bh)\bigr\rvert^2\\
            &=\sum_{\bh}\Biggl(\sum_{m_2}\bigl\lvert F_2(m_2,\bh)\bigr\rvert^2\Biggr)\Biggl(\sum_{m_2^{\prime}}\bigl\lvert F_2(m_2^{\prime},\bh)\bigr\rvert^2\Biggr)\\
            &=\sum_{\bh}\Biggl(\sum_{m_2}\bigl\lvert F_2(m_2,\bh)\bigr\rvert^2\Biggr)^2.
        \end{split}
    \end{equation}

    By \eqref{eqs: F2_4}, \eqref{eqs: F2_3} and Parseval's identity, we have
    \begin{equation}\label{eqs: F2_5}
        \lVert \overline{F_2}(m_2,\bh)F_2(m_2^{\prime},\bh) \rVert_{l^2_{m_2,m_2^{\prime},\bh}}^2=\E_{\bz^{\prime}}\biggl\lvert\E_{x_2,\bz}\Delta_{\bz^{\prime}}f_2(z_1,x_2+z_2)\overline{\Delta_{\bz^{\prime}}f_2}(z_1,z_2)\biggr\rvert^2.
    \end{equation}

    Note that the inner sum of \eqref{eqs: F2_5} is nonnegative since
    \begin{equation*}
        \begin{split}
            \E_{x_2,\bz}\Delta_{\bz^{\prime}}f_2(z_1,x_2+z_2)\overline{\Delta_{\bz^{\prime}}f_2}(z_1,z_2)&=\E_{x_2,\bz}\Delta_{\bz^{\prime}}f_2(z_1,x_2)\overline{\Delta_{\bz^{\prime}}f_2}(z_1,z_2)\\
            &=\E_{z_1}\biggl(\E_{x_2}\Delta_{\bz^{\prime}}f_2(z_1,x_2)\biggr)\biggl(\E_{z_2}\overline{\Delta_{\bz^{\prime}}f_2}(z_1,z_2)\biggr)\\
            &=\E_{z_1}\biggl\lvert\E_{x_2}\Delta_{\bz^{\prime}}f_2(z_1,x_2)\biggr\rvert^2\\
            &\geq 0
        \end{split}
    \end{equation*}

    Nonnegativity of the inner sum of \eqref{eqs: F2_5} and the $1$-boundedness of $f_2$ allow us to deduce that
    \begin{equation}\label{eqs: F2_6}
        \begin{split}
            \lVert \overline{F_2}(m_2,\bh)F_2(m_2^{\prime},\bh) \rVert_{l^2_{m_2,m_2^{\prime},\bh}}^2&\leq\E_{\bz^{\prime}}\E_{x_2,\bz}\overline{\Delta_{\bz^{\prime}}f_2}(z_1,x_2+z_2)\Delta_{\bz^{\prime}}f_2(z_1,z_2)\\
            &=\E_{\bz}\E_{x_2,\bz^{\prime}}\Delta_{(0,x_2)}\Delta_{\bz^{\prime}}f_2(z_1,z_2)\\
            &=\lVert f_2 \rVert_{0\times\Fp,\Fp^2}^4\\
            &\leq \lVert f_2 \rVert_{U^2(0\times\Fp)}^4.
        \end{split}
    \end{equation}
    For the first line of \eqref{eqs: F2_6}, we switch the complex conjugates so that it is easier to recognize the formula as a Gowers norm, and this is valid thanks to the nonnegativity again. For the last line of \eqref{eqs: F2_6}, we use the fact that restricting to subgroup would not decrease Gowers norms.
\end{proof}

Now, we have enough ingredients to prove Theorem \ref{thm: PET}.

\begin{proof}[Proof of Theorem \ref{thm: PET}]
    We simply need to put Propositions \ref{prop: PET}, \ref{prop: Fourier1} and \ref{prop: Fourier2} together to obtain
    \begin{equation*}
        \begin{split}
            &\biggl\lvert\E_{x_1,x_2,y} f_0(x_1,x_2) f_1(x_1+P(y),x_2) f_2(x_1,x_2+Q(y))\biggr\rvert\\
        &\leq\lVert f_0 \rVert_2 \lVert F_1(n_1,\bh) \rVert_{l^2_{n_1,\bh}}^{1/2}\lVert F_2(m_2,\bh) \rVert_{l^2_{m_2,\bh}}^{1/4}\lVert \overline{F_2}(m_2^{\prime},\bh)F_2(m_2^{\dprime},\bh) \rVert_{l^2_{m_2^{\prime},m_2^{\dprime},\bh}}^{1/8}\biggl(\frac{1}{p^6}|Y(\Fp)|\biggr)^{1/16}\\
        &\leq \lVert f_0 \rVert_2 \lVert f_1 \rVert_4 \lVert f_2 \rVert_4^{\frac{1}{2}} \lVert f_2 \rVert_{U^2(0\times \Fp)}^{\frac{1}{4}} \biggl(\frac{1}{p^6}|Y(\Fp)|\biggr)^{\frac{1}{16}}
        \end{split}
    \end{equation*}
\end{proof}

For the purpose of induction, we also need the following ``dual version" of Theorem \ref{thm: PET}.

\begin{thm}\label{thm: PETdual}
    Let $P(t),Q(t)\in \Q(t)$ be rational functions over $\Q$, and $Y(\Fp)$ be the $\Fp$-points of the Roth variety for corners associated to $P(t),Q(t)$. Then, we have the following Gowers norm control
    \begin{multline*}
        \biggl\lvert\E_{x_1,x_2,y} f_0(x_1,x_2) f_1(x_1+P(y),x_2) f_2(x_1,x_2+Q(y))\biggr\rvert\\
        \leq \lVert f_0 \rVert_2 \lVert f_1 \rVert_4^{\frac{1}{2}} \lVert f_1 \rVert_{U^2(\Fp\times 0)}^{\frac{1}{4}}\lVert f_2 \rVert_4 \biggl(\frac{1}{p^6}|Y(\Fp)|\biggr)^{\frac{1}{16}}
    \end{multline*}
    for all $1$-bounded functions $f_0,f_1,f_2:\Fp^2\longrightarrow \C$.
\end{thm}

\begin{proof}
    One could repeat the proof of Theorem \ref{thm: PET} and reverse the roles of $f_1,f_2$ in the order of applying Cauchy--Schwarz inequality. Alternatively, this theorem could be deduced from the Theorem \ref{thm: PET} directly. Let $g_0(x_1,x_2)=f_0(x_2,x_1)$, $g_1(x_1,x_2)=f_1(x_2,x_1)$ and $g_2(x_1,x_2)=f_2(x_2,x_1)$, applying Theorem \ref{thm: PET} to the functions $g_0,g_2,g_1$, we have
    \begin{equation*}
        \begin{split}
            &\biggl\lvert\E_{x_1,x_2,y} f_0(x_1,x_2) f_1(x_1+P(y),x_2) f_2(x_1,x_2+Q(y))\biggr\rvert\\
            &=\biggl\lvert\E_{x_1,x_2,y} g_0(x_2,x_1) g_2(x_2+Q(y),x_1) g_1(x_2,x_1+P(y))\biggr\rvert\\
            &\leq \lVert g_0 \rVert_2 \lVert g_2 \rVert_4 \lVert g_1 \rVert_4^{\frac{1}{2}} \lVert g_1 \rVert_{U^2(0\times \Fp)}^{\frac{1}{4}} \biggl(\frac{1}{p^6}|Y(\Fp)|\biggr)^{\frac{1}{16}}
        \end{split}
    \end{equation*}
    Note that the Roth variety is symmetry in $P(t),Q(t)$ and $\lVert g_1 \rVert_{U^2(0\times \Fp)}=\lVert f_1 \rVert_{U^2(\Fp\times 0)}$.
\end{proof}

\section{Dimension estimates}\label{sec: dim}

In this section, we would prove the point-counting estimates for Roth variety in Theorem \ref{thm: dim}. Heuristically, while $P(t),Q(t)$ and constant function $1$ are linearly independent, the defining equations of the Roth variety for corners associated to $P(t),Q(t)$ are essentially independent, hence the Roth variety has expected dimension $6$.

We start with a preliminary proposition regarding certain determinant. For Proposition \ref{prop: Jac}, we adopt the convention that a blank entry in a matrix is understood to be zero.

\begin{prop}\label{prop: Jac}
    Let $P(t),Q(t)\in \Q(t)$ be rational functions over $\Q$, and $Y$ be the Roth variety for corners associated to $P(t),Q(t)$. Let $J$ be the determinant of the first ten columns of the Jacobian matrix of the defining equations of $Y$, that is
    \begin{equation*}
        \begin{vmatrix}
        P'(y_1)& &-P'(y_3)& & & & & &-P'(y_9)& \\
        &P'(y_2) & &-P'(y_4) & & & & & &-P'(y_{10}) \\
        & & & &P'(y_5) & &-P'(y_7) & & & \\
        & & & & &P'(y_6) & &-P'(y_8) & & \\
        Q'(y_1)&-Q'(y_2) & & & & & & &-Q'(y_9)&Q'(y_{10}) \\
        & &Q'(y_3)& & & &-Q'(y_7) & & & \\
        & & &Q'(y_4) & & & &-Q'(y_8) & & \\
        & & & &Q'(y_5) &-Q'(y_6) & & & & \\
        & & & & & & & &P'(y_9)&-P'(y_{10}) \\
        & & & & & & & &Q'(y_9)&-Q'(y_{10}) 
    \end{vmatrix}.
    \end{equation*}
    Then, we have
    \begin{equation*}
    \begin{split}
        &J=\Bigl(P'(y_1)Q'(y_2)Q'(y_3)P'(y_4)Q'(y_5)P'(y_6)P'(y_7)Q'(y_8)\\
        &\quad\quad\quad\quad\quad -Q'(y_1)P'(y_2)P'(y_3)Q'(y_4)P'(y_5)Q'(y_6)Q'(y_7)P'(y_8)\Bigr)\\
        &\quad\quad\times\Bigl(-P'(y_9)Q'(y_{10})+Q'(y_9)P'(y_{10})\Bigr).
    \end{split}        
    \end{equation*}
\end{prop}

\begin{proof}
    Note that the first eight entries of the ninth and tenth rows of the Jacobian matrix are entirely zeroes, hence we can expand $J$ as
    \begin{equation}\label{eqs: Jac1}
        \begin{split}
            &\begin{vmatrix}
                P'(y_9)&-P'(y_{10}) \\
                Q'(y_9)&-Q'(y_{10}) 
            \end{vmatrix}
            \begin{vmatrix}
                P'(y_1)& &-P'(y_3)& & & & & \\
        &P'(y_2) & &-P'(y_4) & & & &  \\
        & & & &P'(y_5) & &-P'(y_7) & \\
        & & & & &P'(y_6) & &-P'(y_8) \\
        Q'(y_1)&-Q'(y_2) & & & & & & \\
        & &Q'(y_3)& & & &-Q'(y_7) & \\
        & & &Q'(y_4) & & & &-Q'(y_8) \\
        & & & &Q'(y_5) &-Q'(y_6) & & \\
            \end{vmatrix}\\
            &=\Bigl(-P'(y_9)Q'(y_{10})+Q'(y_9)P'(y_{10})\Bigr)J_1
        \end{split}
    \end{equation}
    where $J_1$ is defined to be the determinant of the $8\times 8$ matrix in the first line of \eqref{eqs: Jac1}.

    Next, expand $J_1$ with respect to the first column to obtain
    \begin{equation*}
        \begin{split}
            J_1&=P'(y_1)
            \begin{vmatrix}
        P'(y_2) & &-P'(y_4) & & & &  \\
         & & &P'(y_5) & &-P'(y_7) & \\
         & & & &P'(y_6) & &-P'(y_8) \\
        -Q'(y_2) & & & & & & \\
         &Q'(y_3)& & & &-Q'(y_7) & \\
         & &Q'(y_4) & & & &-Q'(y_8) \\
         & & &Q'(y_5) &-Q'(y_6) & & \\
            \end{vmatrix}\\
            &\quad\quad+Q'(y_1)
            \begin{vmatrix}
                 &-P'(y_3)& & & & & \\
        P'(y_2) & &-P'(y_4) & & & &  \\
         & & &P'(y_5) & &-P'(y_7) & \\
         & & & &P'(y_6) & &-P'(y_8) \\
         &Q'(y_3)& & & &-Q'(y_7) & \\
         & &Q'(y_4) & & & &-Q'(y_8) \\
         & & &Q'(y_5) &-Q'(y_6) & & \\
            \end{vmatrix}\\
            &\eqqcolon A+B
        \end{split}
    \end{equation*}
    where $A,B$ are defined to be the two terms in the cofactor expansion of $J_1$.

    The $7\times 7$ matrices in $A,B$ are fairly sparse. For each of them, we could argue combinatorially that there is only one permutation of $\{1,2,...,7\}$ such that the product of entries is nonzero, and evaluate the determinant, but perhaps it would be clearer if we record the direct computations via repeated use of cofactor expansions of determinants as an alternative approach here.

    We start with the quantity $A$. Expanding the determinant with respect to the fourth row, the expression simplifies to
    \begin{equation*}
    P'(y_1)Q'(y_2)
        \begin{vmatrix}
          &-P'(y_4) & & & &  \\
          & &P'(y_5) & &-P'(y_7) & \\
          & & &P'(y_6) & &-P'(y_8) \\
         Q'(y_3)& & & &-Q'(y_7) & \\
          &Q'(y_4) & & & &-Q'(y_8) \\
          & &Q'(y_5) &-Q'(y_6) & & 
        \end{vmatrix}.
    \end{equation*}

    Expanding the determinant with respect to the first column, the expression simplifies to
    \begin{equation*}
    -P'(y_1)Q'(y_2)Q'(y_3)
        \begin{vmatrix}
          -P'(y_4) & & & &  \\
           &P'(y_5) & &-P'(y_7) & \\
           & &P'(y_6) & &-P'(y_8) \\
          Q'(y_4) & & & &-Q'(y_8) \\
           &Q'(y_5) &-Q'(y_6) & & 
        \end{vmatrix}.
    \end{equation*}

    Expanding the determinant with respect to the first row, the expression simplifies to
    \begin{equation*}
    P'(y_1)Q'(y_2)Q'(y_3)P'(y_4)
        \begin{vmatrix}
           P'(y_5) & &-P'(y_7) & \\
            &P'(y_6) & &-P'(y_8) \\
           & & &-Q'(y_8) \\
           Q'(y_5) &-Q'(y_6) & & 
        \end{vmatrix}.
    \end{equation*}

    Expanding the determinant with respect to the third column, the expression simplifies to
    \begin{equation*}
    -P'(y_1)Q'(y_2)Q'(y_3)P'(y_4)P'(y_7)
        \begin{vmatrix}
            &P'(y_6)  &-P'(y_8) \\
           &  &-Q'(y_8) \\
           Q'(y_5) &-Q'(y_6)  & 
        \end{vmatrix}.
    \end{equation*}

    Expanding the determinant with respect to the first column, the expression simplifies to
    \begin{equation*}
    -P'(y_1)Q'(y_2)Q'(y_3)P'(y_4)Q'(y_5)P'(y_7)
        \begin{vmatrix}
            P'(y_6)  &-P'(y_8) \\
             &-Q'(y_8) 
        \end{vmatrix}.
    \end{equation*}

    Expanding the $2\times 2$ determinant, the expression simplifies to
    \begin{equation*}    P'(y_1)Q'(y_2)Q'(y_3)P'(y_4)Q'(y_5)P'(y_6)P'(y_7)Q'(y_8)
    \end{equation*}

    Next, we repeat the same process for $B$. Expanding the determinant with respect to the first column, the expression simplifies to
    \begin{equation*}
        -Q'(y_1)P'(y_2)
        \begin{vmatrix}
                 -P'(y_3)& & & & & \\
          & &P'(y_5) & &-P'(y_7) & \\
          & & &P'(y_6) & &-P'(y_8) \\
         Q'(y_3)& & & &-Q'(y_7) & \\
          &Q'(y_4) & & & &-Q'(y_8) \\
          & &Q'(y_5) &-Q'(y_6) & & \\
            \end{vmatrix}.
    \end{equation*}

    Expanding the determinant with respect to the first row, the expression simplifies to
    \begin{equation*}
        Q'(y_1)P'(y_2)P'(y_3)
        \begin{vmatrix}
           &P'(y_5) & &-P'(y_7) & \\
           & &P'(y_6) & &-P'(y_8) \\
          & & &-Q'(y_7) & \\
          Q'(y_4) & & & &-Q'(y_8) \\
           &Q'(y_5) &-Q'(y_6) & & \\
            \end{vmatrix}.
    \end{equation*}

    Expanding the determinant with respect to the first column, the expression simplifies to
    \begin{equation*}
        -Q'(y_1)P'(y_2)P'(y_3)Q'(y_4)
        \begin{vmatrix}
           P'(y_5) & &-P'(y_7) & \\
            &P'(y_6) & &-P'(y_8) \\
           & &-Q'(y_7) & \\
           Q'(y_5) &-Q'(y_6) & & \\
            \end{vmatrix}.
    \end{equation*}

    Expanding the determinant with respect to the last column, the expression simplifies to
    \begin{equation*}
        Q'(y_1)P'(y_2)P'(y_3)Q'(y_4)P'(y_8)
        \begin{vmatrix}
           P'(y_5) & &-P'(y_7)  \\
           & &-Q'(y_7)  \\
           Q'(y_5) &-Q'(y_6) &  
            \end{vmatrix}.
    \end{equation*}

    Expanding the determinant with respect to the second row, the expression simplifies to
    \begin{equation*}
        Q'(y_1)P'(y_2)P'(y_3)Q'(y_4)Q'(y_7)P'(y_8)
        \begin{vmatrix}
           P'(y_5) &   \\
           Q'(y_5) &-Q'(y_6)   
            \end{vmatrix}.
    \end{equation*}

    Expanding the $2\times 2$ determinant, the expression simplifies to
    \begin{equation*}
        -Q'(y_1)P'(y_2)P'(y_3)Q'(y_4)P'(y_5)Q'(y_6)Q'(y_7)P'(y_8),
    \end{equation*}
    this completes the proof.
\end{proof}

The main theorem in this section is Theorem \ref{thm: dim}. The proof of Theorem \ref{thm: dim} is quite technical and long, but the main idea is actually fairly intuitive and direct. The strategy is similar to the one that Hong and the author \cite{hong-lim} used for the Roth variety for one dimensional progressions. We would decompose the Roth variety into smooth part and singular part. The smooth part could be dealt with complex analytic technique, while the singular part has extra constraints for free. Furthermore, we would repeat this process several times, hence we would have ``smooth part of the singular part" and ``singular part of the singular part" and so on. After enough decompositions, there would be sufficient constraints that only depend on a few variables, and this is very helpful for point-counting estimates.

\begin{thm}\label{thm: dim}
    Let $P(t),Q(t)\in \Q(t)$ be rational functions over $\Q$ such that $P(t),Q(t)$ and constant function $1$ are linearly independent over $\Q$, and $Y(\Fp)$ be the $\Fp$-points of the Roth variety for corners associated to $P(t),Q(t)$. Then, we have the following point-counting estimate
    \begin{equation*}
        |Y(\Fp)|\ll p^6.
    \end{equation*}
\end{thm}

\begin{proof}
We would break down the proof into eight steps. Here is a roadmap of the proof. The main steps are Steps $1,2,3,4,6$.
\begin{itemize}
    \item Step $1$: Bounding the generic part of the Roth variety via complex analytic techniques. The readers are recommended to glance through this step and skip the technical details.
    \item Step $2$: Decomposing the Roth variety further to divide and conquer. This step is mostly the definitions and constructions.
    \item Steps $3,4$: Showing that certain constraint is nontrivial. These two steps are the key calculations.
    \item Step $6$: Bounding the special part of the Roth variety. This step is quite representative and the same method would also be used in Steps $5,7,8$.
    \item Steps $5,7,8$: Dealing with the degenerate cases. These three steps can be skipped on a first reading.
\end{itemize}
While reading the proof, the readers are encouraged to take $P,Q$ as polynomials, in fact, taking $P(t)=t,Q(t)=t^2$ would be a fairly illustrative example.

    We would start with some preliminary decomposition of $Y$. Let $J(y_1,y_2,...,y_{10})$ be the determinant of the first ten columns of the Jacobian matrix of the defining equations of $Y$, then by Proposition \ref{prop: Jac}, we have
    \begin{equation*}
        \begin{split}
        &J(y_1,y_2,y_3,y_4,y_5,y_6,y_7,y_8,y_9,y_{10})\\
        &=\Bigl(P'(y_1)Q'(y_2)Q'(y_3)P'(y_4)Q'(y_5)P'(y_6)P'(y_7)Q'(y_8)\\
        &\quad\quad\quad\quad\quad -Q'(y_1)P'(y_2)P'(y_3)Q'(y_4)P'(y_5)Q'(y_6)Q'(y_7)P'(y_8)\Bigr)\\
        &\quad\quad\times\Bigl(-P'(y_9)Q'(y_{10})+Q'(y_9)P'(y_{10})\Bigr).
    \end{split}
    \end{equation*}
    
    Define
    \begin{equation*}
        Y_{\mathrm{gen}}=Y\cap \{(y_1,y_2,...,y_{16})\in \mathbb{A}^{16}:J(y_1,y_2,y_3,y_4,y_5,y_6,y_7,y_8,y_9,y_{10})\neq 0\}
    \end{equation*}
    and
    \begin{equation*}
        Y_{\mathrm{sp}}=Y\cap \{(y_1,y_2,...,y_{16})\in \mathbb{A}^{16}:J(y_1,y_2,y_3,y_4,y_5,y_6,y_7,y_8,y_9,y_{10})=0\}
    \end{equation*}
    where $\mathbb{A}$ is the affine space.

    It makes sense to define $Y_{\mathrm{gen}}$ and $Y_{\mathrm{sp}}$ over any fields, hence the specific field is not mentioned in the definitions, and it would be convenient to work over various fields. Note that we have $Y(\mathbb{F}_p)=Y_{\mathrm{gen}}(\mathbb{F}_p)\cup Y_{\mathrm{sp}}(\mathbb{F}_p)$. Informally, the variety $Y_{\mathrm{gen}}$ is a generic part of $Y$ which lies in the smooth component of $Y$, and $Y_{\mathrm{sp}}$ is a special subvariety of $Y$ cut out by $J$.\\

    \underline{\textbf{Step $1$: Bounding $|Y_{\mathrm{gen}}(\mathbb{F}_p)|$ via complex analytic techniques}}

    In this step, we would bound $|Y_{\mathrm{gen}}(\mathbb{F}_p)|$ via complex analytic techniques. This step follows the same complex analytic method that Hong and the author \cite[Theorem 4.1]{hong-lim} use to bound the size of the $\Fp$-points on the Roth variety for one dimensional progressions. The main point of this step is saying various notions of dimensions are the same, and the readers are recommended to glance through this step and skip the technical details.

    The associated complex analytic space of $Y_{\mathrm{gen}}(\mathbb{C})$ is a $6$-dimensional complex manifold since the Jacobian matrices have full rank and we can apply the implicit function theorem. By Noether normalization, there exists a finite surjective map $Y_{\mathrm{gen}}(\mathbb{C})\longrightarrow \mathbb{A}^d(\mathbb{C})$ between varieties, where $d$ is the dimension of $Y_{\mathrm{gen}}(\mathbb{C})$ as a variety and $\mathbb{A}^d$ is the $d$-dimensional affine space. Regard the map $Y_{\mathrm{gen}}(\mathbb{C})\longrightarrow \mathbb{A}^d(\mathbb{C})$ as a finite holomorphic map between complex manifolds. By the first theorem in \cite[Chapter 5, Section 4.1]{MR0755331}, we deduce that the dimension of $Y_{\mathrm{gen}}(\mathbb{C})$ as a variety is $3$. (We compare the dimension of a variety over $\mathbb{C}$ and the dimension of its associated complex analytic space via \textit{ad hoc} argument here, for general dimension theory of complex analytic spaces, see \cite[Appendix B]{MR0463157} and \cite[Chapter 5, Section 1]{MR0755331}.)
    
    Consider an affine $\mathbb{Z}$-model of $Y_{\mathrm{gen}}(\mathbb{C})$ by clearing out the denominators in the defining equations of the Roth variety and adding one auxiliary variable $z$ and one equation to require that the denominators are nonzero. To be more specific, write $P(t)=a(t)/b(t)$ and $Q(t)=c(t)/d(t)$ with $a(t),b(t),c(t),d(t)\in \mathbb{Z}[t]$, let $I\subset \mathbb{Z}[y_1,y_2,...,y_{16},z]$ be the ideal generated by the following polynomials:
    \begin{gather*}
        \Bigl(b(y_1)b(y_{3})b(y_{9})b(y_{11})\Bigr)\Bigl(P(y_1)-P(y_3)-P(y_9)+P(y_{11})\Bigr)\\
        \Bigl(b(y_2)b(y_{4})b(y_{10})b(y_{12})\Bigr)\Bigl(P(y_2)-P(y_4)-P(y_{10})+P(y_{12})\Bigr)\\
        \Bigl(b(y_5)b(y_{7})b(y_{13})b(y_{15})\Bigr)\Bigl(P(y_5)-P(y_7)-P(y_{13})+P(y_{15})\Bigr)\\
        \Bigl(b(y_6)b(y_{8})b(y_{14})b(y_{16})\Bigr)\Bigl(P(y_6)-P(y_8)-P(y_{14})+P(y_{16})\Bigr)\\
        \Bigl(d(y_1)d(y_{2})d(y_{9})d(y_{10})\Bigr)\Bigl(Q(y_1)-Q(y_2)-Q(y_{9})+Q(y_{10})\Bigr)\\
        \Bigl(d(y_3)d(y_{7})d(y_{11})d(y_{15})\Bigr)\Bigl(Q(y_3)-Q(y_7)-Q(y_{11})+Q(y_{15})\Bigr)\\
        \Bigl(d(y_4)d(y_{8})d(y_{12})d(y_{16})\Bigr)\Bigl(Q(y_4)-Q(y_8)-Q(y_{12})+Q(y_{16})\Bigr)\\
        \Bigl(d(y_5)d(y_{6})d(y_{13})d(y_{14})\Bigr)\Bigl(Q(y_5)-Q(y_6)-Q(y_{13})+Q(y_{14})\Bigr)\\
        \Bigl(\prod_{i=9}^{16}b(y_i)\Bigr)\Bigl(P(y_{9})-P(y_{10})-P(y_{11})+P(y_{12})-P(y_{13})+P(y_{14})+P(y_{15})-P(y_{16})\Bigr)\\
        \Bigl(\prod_{i=9}^{16}d(y_i)\Bigr)\Bigl(Q(y_{9})-Q(y_{10})-Q(y_{11})+Q(y_{12})-Q(y_{13})+Q(y_{14})+Q(y_{15})-Q(y_{16})\Bigr)\\
        z\Bigl(\prod_{i=1}^{16}b(y_i)d(y_i)\Bigr)\biggl[\Bigl(\prod_{i=1}^{10}b(y_i)^2d(y_i)^2\Bigr)J(y_1,y_2,...y_{10})\biggr]-1
    \end{gather*}
    The first tenth generators of the ideal $I$ are the denominator free version of the defining equations of the Roth variety. Including the last generator of the ideal $I$ is a localization trick in algebraic geometry; it ensures $J(y_1,y_2,...,y_{16})$ and the denominators of $P$ and $Q$ are nonzero, but it does not impose any unwanted constraints since the auxiliary variable $z$ is uniquely determined by $y_1,y_2,...,y_{16}$.  
    
    Define affine $\mathbb{Z}$-model $Y_{\mathrm{gen}}=\mathrm{Spec}\; \mathbb{Z}[y_1,y_2,...,y_{16},z]/I$. Note that the $\mathbb{C}$-points of $Y_{\mathrm{gen}}\times_\mathbb{Z} \mathbb{C}$ is isomorphic to $Y_{\mathrm{gen}}(\mathbb{C})$, hence we have $\dim_{\mathrm{Sch}}Y_{\mathrm{gen}}\times_\mathbb{Z} \mathbb{C}=\dim_{\mathrm{Var}}Y_{\mathrm{gen}}(\mathbb{C})=6$. By \cite[Chapter 2, Exercise 3.20.(f)]{MR0463157}, we have the invariance of dimension under base change, i.e., $\dim_{\mathrm{Sch}}Y_{\mathrm{gen}}\times_\mathbb{Z} \mathbb{Q}=\dim_{\mathrm{Sch}}Y_{\mathrm{gen}}\times_\mathbb{Z} \mathbb{C}=6$. (Note that the integral assumption in \cite[Chapter 2, Exercise 3.20.(f)]{MR0463157} can be removed if one does not require the scheme after base change to be equidimensional.)
    
    Note that the structure map $Y_{\mathrm{gen}}=\mathrm{Spec}\; \mathbb{Z}[y_1,y_2,...,y_{16},z]/I \longrightarrow \mathrm{Spec}\;\mathbb{Z}$ is of finite type and the base $\mathrm{Spec}\;\mathbb{Z}$ is irreducible, by \cite[\href{https://stacks.math.columbia.edu/tag/05F6}{Section 05F6, Lemma 37.30.1}]{stacks-project}, we deduce that the dimension of the generic fiber is the same as mod $p$ fibers for all but finitely many $p$, that is, $\dim_{\mathrm{Sch}}Y_{\mathrm{gen}}\times_\mathbb{Z} \mathbb{Q}=\dim_{\mathrm{Sch}}Y_{\mathrm{gen}}\times_\mathbb{Z} \mathbb{F}_p=6$ for all but finitely many $p$. Once again, by \cite[Chapter 2, Exercise 3.20.(f)]{MR0463157}, we have $\dim_{\mathrm{Sch}}Y_{\mathrm{gen}}\times_\mathbb{Z} \mathbb{F}_p=\dim_{\mathrm{Sch}}Y_{\mathrm{gen}}\times_\mathbb{Z} \overline{\mathbb{F}}_p$. Hence, we have $\dim_{\mathrm{Var}}Y_{\mathrm{gen}}(\overline{\mathbb{F}}_p)=6$ for all but finitely many $p$. Now, the Lang--Weil bound \cite{MR0065218} (also see \cite[Lemma 1]{Tao})  implies that $|Y_{\mathrm{gen}}(\mathbb{F}_p)|\ll p^6$. (Note that one does not need the irreducibility of the variety.)
    
    (\emph{As a side note, we pass to complex numbers since it is convenient to work with algebraic geometry over characteristic zero and rational functions over complex numbers. That being said, the machinery of algebraic geometry is robust enough for one to only work over finite fields. Therefore, the approach we chose is just one possible way to proceed.})\\

    \underline{\textbf{Step $2$: Decomposing $Y_{\mathrm{sp}}$}}

    This step is mostly the definitions and constructions. In this step, we would decompose $Y_{\mathrm{sp}}$ and isolate the main difficulty.

    Let 
    \begin{equation*}
        R(y)=\frac{P'(y)}{Q'(y)}
    \end{equation*}
    and rewrite $J(y_1,y_2,...,y_{10})$ as
    \begin{equation*}
        \begin{split}
        &\Bigl(-P'(y_9)Q'(y_{10})+Q'(y_9)P'(y_{10})\Bigr)\\
        &\times\Bigl(Q'(y_1)Q'(y_2)Q'(y_3)Q'(y_4)Q'(y_5)Q'(y_6)Q'(y_7)Q'(y_8)\Bigr)\\
        &\times\Bigl(R(y_1)R(y_4)R(y_6)R(y_7)-R(y_2)R(y_3)R(y_5)R(y_8)\Bigr).
    \end{split}
    \end{equation*}

    Define
    \begin{gather*}
        Z=Y\cap\{(y_1,y_2,...,y_{16})\in \mathbb{A}^{16}:R(y_1)R(y_4)R(y_6)R(y_7)-R(y_2)R(y_3)R(y_5)R(y_8)=0\}\\
        Y_{i}=Y\cap\{(y_1,y_2,...,y_{16})\in \mathbb{A}^{16}:Q'(y_i)=0\}\quad\text{for}\; 1\leq i \leq 8\\
        Y_{9,10}=Y\cap\{(y_1,y_2,...,y_{16})\in \mathbb{A}^{16}:-P'(y_9)Q'(y_{10})+Q'(y_9)P'(y_{10})=0\}.
    \end{gather*}
    Note that $Y_{\mathrm{sp}}(\Fp)=Z(\Fp)\cup Y_1(\Fp)\cup Y_2(\Fp)\cup...\cup Y_8(\Fp)\cup Y_{9,10}(\Fp)$. Informally, the varieties $Y_1,Y_2,...,Y_8,Y_{9,10}$ are of low dimension due to the strong constraints $Q'(y_i)=0$ for $1\leq i\leq 8$ and $-P'(y_9)Q'(y_{10})+Q'(y_9)P'(y_{10})=0$, and they can be handled more easily later at Steps $7,8$. The main difficulty lies in the variety $Z$. 
    
    To bound the size of $Z(\Fp)$, we would define an auxiliary variety $Z'$. Let $Z'$ be the variety cut out by the following three equations in eight variables $y_1,y_2,...,y_8$.
    \begin{gather}
        P(y_{1})-P(y_{2})-P(y_{3})+P(y_{4})-P(y_{5})+P(y_{6})+P(y_{7})-P(y_{8})=0\label{eqs: dim2}\\
        Q(y_{1})-Q(y_{2})-Q(y_{3})+Q(y_{4})-Q(y_{5})+Q(y_{6})+Q(y_{7})-Q(y_{8})=0\label{eqs: dim3}\\
        R(y_1)R(y_4)R(y_6)R(y_7)-R(y_2)R(y_3)R(y_5)R(y_8)=0\label{eqs: dim4}
    \end{gather}
    Note that $Z'(\Fp)$ lives in $\Fp^8$. The first two defining equations of $Z'$ do not come out of nowhere; the points on the Roth variety $Y$ actually satisfy those two equations. To see this, we manipulate the defining equations of the Roth variety algebraically; consider adding the equations $\eqref{eqs: roth1}-\eqref{eqs: roth2}-\eqref{eqs: roth3}+\eqref{eqs: roth4}+\eqref{eqs: roth9}$ and $\eqref{eqs: roth5}-\eqref{eqs: roth6}+\eqref{eqs: roth7}-\eqref{eqs: roth8}+\eqref{eqs: roth10}$. Therefore, the varieties $Z$ and $Z'$ are closely related: for any $(y_1,y_2,...,y_{16})\in Z(\Fp)$, we have $(y_1,y_2,...,y_8)\in Z'(\Fp)$.

    The main goal would be showing $|Z'(\Fp)|\ll p^5$. To do that, once again we decompose $Z'$ into smooth part and singular parts. Let $J_{Z'}(y_1,y_4,y_6)$ be the determinant of the first, fourth and sixth columns of the Jacobian matrix of the defining equations of $Z'$, that is
    \begin{equation*}
        J_{Z'}(y_1,y_4,y_6)=
        \begin{vmatrix}
            P'(y_1)& P'(y_4)& P'(y_6)\\
            Q'(y_1)& Q'(y_4)& Q'(y_6)\\
            R'(y_1)R(y_4)R(y_6)R(y_7)& R(y_1)R'(y_4)R(y_6)R(y_7)& R(y_1)R(y_4)R'(y_6)R(y_7)\\
        \end{vmatrix}.
    \end{equation*}
    Let $S(y)$ be the logarithmic derivative of $R(y)$, i.e.,
    \begin{equation*}
        S(y)=\frac{R'(y)}{R(y)}
    \end{equation*}
    and
    \begin{equation*}
    D(y_1,y_4,y_6)=
        \begin{vmatrix}
            P'(y_1)& P'(y_4)& P'(y_6)\\
            Q'(y_1)& Q'(y_4)& Q'(y_6)\\
            S(y_1)& S(y_4)& S(y_6)\\
        \end{vmatrix}
    \end{equation*}
    so that we could rewrite
    \begin{equation*}
        J_{Z'}(y_1,y_4,y_6)=R(y_1)R(y_4)R(y_6)R(y_7)D(y_1,y_4,y_6).
    \end{equation*}

    Define
    \begin{gather*}
        Z'_{\mathrm{gen}}=Z'\cap \{(y_1,y_2,...,y_8)\in \mathbb{A}^8:J_{Z'}(y_1,y_4,y_6)\neq 0\}\\
        Z'_{i}=Z'\cap \{(y_1,y_2,...,y_8)\in \mathbb{A}^8:R(y_i)=0\}\quad\text{for}\;i=1,4,6,7\\
        Z'_{\mathrm{sp}}=Z'\cap \{(y_1,y_2,...,y_8)\in \mathbb{A}^8:D(y_1,y_4,y_6)= 0\}.
    \end{gather*}
    Note that $Z'(\Fp)=Z'_{\mathrm{gen}}(\Fp)\cup Z'_{1}(\Fp)\cup Z'_{4}(\Fp)\cup Z'_{6}(\Fp)\cup Z'_{7}(\Fp)\cup Z'_{\mathrm{sp}}(\Fp)$. The variety $Z'_{\mathrm{gen}}$ is a generic part of $Z'$ that lies in the smooth component of $Z'$. The associated complex analytic space of $Z'_{\mathrm{gen}}(\mathbb{C})$ is a $5$-dimensional complex manifold since the Jacobian matrices have full rank and we can apply the implicit function theorem. By similar arguments to Step $1$, we could compare various dimensions and conclude that $|Z'_{\mathrm{gen}}(\Fp)|\ll p^5$. The varieties $Z'_1,Z'_4,Z'_6,Z'_7$ are of low dimension due to the strong constraints $R(y_i)=0$ for $i=1,4,6,7$, and they can be handled more easily later at Step $5$. We would proceed to bound $|Z'_{\mathrm{sp}}(\Fp)|$ in Steps $3,4$.\\

    \underline{\textbf{Step 3: Showing that $D(y_1,y_4,y_6)$ is nonzero}}
    
    In this step, we would show that $D(y_1,y_4,y_6)$ is a nonzero rational function, hence $D(y_1,y_4,y_6)=0$ is a nontrivial constraint for $Z'_{\mathrm{sp}}$. This step crucially relies on the assumption that $P,Q$ and constant function $1$ are linearly independent.
    
    Expanding $D(y_1,y_4,y_6)$ with respect to the first column, we have
    \begin{equation}\label{eqs: dim1}
        \begin{split}
            &D(y_1,y_4,y_6)\\
            &=P'(y_1)
            \begin{vmatrix}
                Q'(y_4)&Q'(y_6)\\
                S(y_4)&S(y_6)
            \end{vmatrix}
            -Q'(y_1)
            \begin{vmatrix}
                P'(y_4)&P'(y_6)\\
                S(y_4)&S(y_6)
            \end{vmatrix}
            +S(y_1)
            \begin{vmatrix}
            P'(y_4)&P'(y_6)\\
                Q'(y_4)&Q'(y_6)
            \end{vmatrix}.
        \end{split}
    \end{equation}

    Note that
    \begin{equation*}
    \begin{split}
                S(y)&=\frac{R'(y)}{R(y)}\\
        &=\frac{1}{R(y)}\biggl(\frac{P'(y)}{Q'(y)}\biggr)'\\
        &=\frac{Q'(y)}{P'(y)}\frac{Q'(y)P''(y)-P'(y)Q''(y)}{(Q'(y))^2}\\
        &=\frac{P''(y)}{P'(y)}-\frac{Q''(y)}{Q'(y)}
    \end{split}
    \end{equation*}
    hence $S(y)$ is the difference of the logarithmic derivatives of $P'(y)$ and $Q'(y)$.

    If the poles of $P'(y_1)$ and $Q'(y_1)$ are different, then $S(y_1)$ would at least have a simple pole since it is the difference of the logarithmic derivatives of $P'(y_1)$ and $Q'(y_1)$. Meanwhile, the poles of $P'(y_1)$ and $Q'(y_1)$ cannot be simple poles since they are derivatives of meromorphic functions. Therefore, in the expression \eqref{eqs: dim1}, the poles of $S(y_1)$ could not be cancelled out by the poles of $P'(y_1)$ and $Q'(y_1)$, and hence $D(y_1,y_4,y_6)$ is nonzero. 

    Similarly, if the zeroes of $P'(y_1)$ and $Q'(y_1)$ are different, then $S(y_1)$ would at least have a simple pole since it is the difference of the logarithmic derivatives of $P'(y_1)$ and $Q'(y_1)$. Once again, the poles of $S(y_1)$ could not be cancelled out by the poles of $P'(y_1)$ and $Q'(y_1)$, and hence $D(y_1,y_4,y_6)$ is nonzero.

    If the poles and zeroes of $P'(y_1)$ and $Q'(y_1)$ are exactly the same, then $P'$ and $Q'$ would only differ up to a scalar, and this would contradict the linear independence of $P,Q$ and constant function $1$.

    As an ending remark, we do have a little detail not checked yet, that is, we need to check that $\bigl\lvert\begin{smallmatrix}
            P'(y_4)&P'(y_6)\\
                Q'(y_4)&Q'(y_6)
            \end{smallmatrix}\bigr\rvert$ 
    is nonzero so that it could not cancel out the poles of $S(y_1)$. Suppose the said determinant is zero, then $P'(y_4)/Q'(y_4)=P'(y_6)/Q'(y_6)$, hence $P'/Q'$ would be a constant and this contradicts the linear independence of $P,Q$ and constant function $1$. This completes the proof of Step $3$.\\

    \underline{\textbf{Step $4$: Bounding $|Z'_{\mathrm{sp}}(\Fp)|$}}

    In this step, we would bound $|Z'_{\mathrm{sp}}(\Fp)|$. The strategy is to use $D(y_1,y_4,y_6)=0$ to constraint variables $y_1,y_4,y_6$ and the equations \eqref{eqs: dim2}, \eqref{eqs: dim3} to constraint variables $y_2,y_3,y_5,y_7,y_8$.
    
    First, we would define auxiliary varieties $X_{a,b}$. For any $1\leq a,b \leq p$, let $X_{a,b}$ be the variety cut out by the following two equations in five variables $y_2,y_3,y_5,y_7,y_8$.
    \begin{gather}
        -P(y_2)-P(y_3)-P(y_5)+P(y_7)-P(y_8)=a\label{eqs: dim5}\\
        -Q(y_2)-Q(y_3)-Q(y_5)+Q(y_7)-Q(y_8)=b\label{eqs: dim6}
    \end{gather}
    Note that $X_{a,b}(\Fp)$ live in $\Fp^5$. We claim that $|X_{a,b}(\Fp)|\ll p^3$ where the implied constant does not depend on $a,b$.

    To prove our claim, we would decompose $X_{a,b}$ into smooth part and singular part. Let $J_{X}(y_7,y_8)$ be the determinant of the last two columns of the Jacobian matrix of the defining equations of $X_{a,b}$, that is
    \begin{equation*}
        J_{X}(y_7,y_8)=
        \begin{vmatrix}
            P'(y_7)&-P'(y_8)\\
            Q'(y_7)&-Q'(y_8)\\
        \end{vmatrix}
        =-P'(y_7)Q'(y_8)+Q'(y_7)P'(y_8).
    \end{equation*}
    
    Define
    \begin{gather*}
        X_{a,b,\mathrm{gen}}=X_{a,b}\cap \{(y_2,y_3,y_5,y_7,y_8)\in\mathbb{A}^5:J_{X}(y_7,y_8)\neq 0\}\\
        X_{a,b,\mathrm{sp}}=X_{a,b}\cap \{(y_2,y_3,y_5,y_7,y_8)\in\mathbb{A}^5:J_{X}(y_7,y_8)=0\}
    \end{gather*}
    Note that $X_{a,b}(\Fp)=X_{a,b,\mathrm{gen}}(\Fp)\cup X_{a,b,\mathrm{sp}}(\Fp)$. 

    To bound $|X_{a,b,\mathrm{gen}}(\Fp)|$, note that the associated complex analytic space of $X_{a,b,\mathrm{gen}}(\mathbb{C})$ is a $3$-dimensional complex manifold since the Jacobian matrices have full rank and we can apply the implicit function theorem. By similar arguments to Step $1$, we could compare various dimensions and conclude that $|X_{a,b,\mathrm{gen}}(\Fp)|\ll p^3$, where the implied constant does not depend on $a,b$. 

    To bound $|X_{a,b,\mathrm{sp}}(\Fp)|$, note that $J_{X}(y_7,y_8)$ is a nonzero rational function due to the linear independence of $P,Q$ and constant function $1$, hence by Lang--Weil bound \cite{MR0065218} (or Schwartz--Zippel lemma), there are at most $O(p)$ choices $y_7,y_8$ such that $J_{X}(y_7,y_8)=0$. For any given $y_7,y_8$, there are $p^2$ choices of $y_2,y_3$. For any given $y_7,y_8,y_2,y_3$, there are $O(1)$ choices $y_5$ such that $y_7,y_8,y_2,y_3,y_5$ satisfy equation \eqref{eqs: dim5}. Hence, we conclude $|X_{a,b,\mathrm{sp}}(\Fp)|\ll p^3$ and hence $|X_{a,b}(\Fp)|\ll p^3$.

    Finally, for $(y_1,y_2,...,y_8)\in Z'_{\mathrm{sp}}(\Fp)$, there are at most $O(p^2)$ choices of $y_1,y_4,y_6$ since $D(y_1,y_4,y_6)$ is a nonzero rational function [Step $3$] and Lang--Weil bound \cite{MR0065218} (or Schwartz--Zippel lemma) would imply the bounds. Fix $y_1,y_4,y_6$, for any $(y_1,y_2,...,y_8)\in Z'_{\mathrm{sp}}(\Fp)$, the tuples $(y_2,y_3,y_5,y_7,y_8)\in X_{a,b}(\Fp)$ for some $a,b$ by equations \eqref{eqs: dim2} and \eqref{eqs: dim3}, hence we have at most $O(p^3)$ choices of $y_2,y_3,y_5,y_7,y_8$. Therefore, we conclude that $|Z'_{\mathrm{sp}}(\Fp)|\ll p^5$.\\

    \underline{\textbf{Step $5$: Bounding $|Z'_{i}(\Fp)|$ for $i=1,4,6,7$}}

    This step is more of a routine step and can be skipped on a first reading.

    We would demonstrate the approach for bounding $|Z'_{1}(\Fp)|$, and the bounds for $|Z'_{4}(\Fp)|$, $|Z'_{6}(\Fp)|$, $|Z'_{7}(\Fp)|$ follow similarly. Recall that 
    \begin{equation*}
        Z'_{1}=Z'\cap \{(y_1,y_2,...,y_8)\in \mathbb{A}^8:R(y_1)=0\}.
    \end{equation*}
\begin{itemize} 
    \item For any $(y_1,y_2,...,y_8)\in Z'_{1}(\Fp)$, there are at most $O(1)$ choices of $y_1$ by equation $R(y_1)=0$.

    \item By equation \eqref{eqs: dim4} and $R(y_1)=0$, we have $R(y_2)R(y_3)R(y_5)R(y_8)=0$, hence for any given $y_1$, there are at most $O(p^3)$ choices of $y_2,y_3,y_5,y_8$.

    \item For any given $y_1,y_2,y_3,y_5,y_8$, there are $p^2$ choices of $y_4,y_6$.

    \item For any given $y_1,y_2,y_3,y_5,y_8,y_4,y_6$, there are at most $O(1)$ choices of $y_7$ by equation \eqref{eqs: dim2}.

    \item Thus, there are at most $O(p^3p^2)=O(p^5)$ choices of $y_1,y_2,y_3,y_5,y_8,y_4,y_6,y_7$. 
\end{itemize}
    Therefore, we conclude that $|Z'_{1}(\Fp)|\ll p^5$. By putting the bounds for $|Z'_{\mathrm{gen}}(\Fp)|,|Z'_{\mathrm{sp}}(\Fp)|$ and $|Z'_{i}(\Fp)|$ for $i=1,4,6,7$, we further conclude that $|Z'(\Fp)|\ll p^5$.\\

    \underline{\textbf{Step $6$: Bounding $|Z(\Fp)|$}}

    This step is quite representative and the same method would also be used in Steps $5,7,8$. Heuristically, when there are enough constraints, we could pick a subset of variables so that the system of equations becomes ``triangular''.
\begin{itemize}
    \item For any $(y_1,y_2,...,y_{16})\in Z(\Fp)$, we have $(y_1,y_2,...,y_8)\in Z'(\Fp)$ by construction of $Z'$, hence there are at most $O(p^5)$ choices of $y_1,y_2,y_3,y_4,y_5,y_6,y_7,y_8$.

    \item For any given $y_1,y_2,y_3,y_4,y_5,y_6,y_7,y_8$, there are $p$ choices of $y_9$.

    \item For any given $y_1,y_2,y_3,y_4,y_5,y_6,y_7,y_8,y_9$, there are at most $O(1)$ choices of $y_{11},y_{10}$ by equations \eqref{eqs: roth1}, \eqref{eqs: roth5}.

    \item For any given $y_1,y_2,y_3,y_4,y_5,y_6,y_7,y_8,y_9,y_{11},y_{10}$, there are at most $O(1)$ choices of $y_{12},y_{15}$ by equations \eqref{eqs: roth2}, \eqref{eqs: roth6}.

    \item For any given $y_1,y_2,y_3,y_4,y_5,y_6,y_7,y_8,y_9,y_{11},y_{10},y_{12},y_{15}$, there are at most $O(1)$ choices of $y_{13},y_{16}$ by equations \eqref{eqs: roth3}, \eqref{eqs: roth7}.

    \item For any given $y_1,y_2,y_3,y_4,y_5,y_6,y_7,y_8,y_9,y_{11},y_{10},y_{12},y_{15},y_{13},y_{16}$, there are at most $O(1)$ choices of $y_{14}$ by equation \eqref{eqs: roth4}.

    \item Thus, there are at most $O(p^5p)=O(p^6)$ choices of $y_1,y_2,y_3,y_4,y_5,y_6,y_7,y_8,y_9,y_{11},\\y_{10},y_{12},y_{15},y_{13},y_{16},y_{14}$. 
\end{itemize}
    Therefore, we conclude that $|Z(\Fp)|\ll p^6$.\\

    \underline{\textbf{Step $7$: Bounding $|Y_{i}(\Fp)|$ for $1\leq i\leq 8$}}

    This step is more of a routine step and can be skipped on a first reading.

    We would demonstrate the approach for bounding $|Y_{1}(\Fp)|$, and the bounds for $|Y_{i}(\Fp)|$ for $2\leq i\leq 8$ follow similarly. Recall that
    \begin{equation*}
        Y_{1}=Y\cap\{(y_1,y_2,...,y_{16})\in \mathbb{A}^{16}:Q'(y_1)=0\}.
    \end{equation*}

\begin{itemize}
    \item For any $(y_1,y_2,...,y_{16})\in Y_{1}(\Fp)$, there are at most $O(1)$ choices of $y_1$ by equation $Q'(y_1)=0$.

    \item For any given $y_1$, there are $p^2$ choices of $y_4,y_6$.

    \item Fix $y_1,y_4,y_6$, for any $(y_1,y_2,...,y_{16})\in Y_{1}(\Fp)$, we have $(y_2,y_3,y_5,y_7,y_8)\in X_{a,b}(\Fp)$ for some $a,b$ by equations \eqref{eqs: dim2}, \eqref{eqs: dim3}. (Note that the defining equations of the Roth variety $Y$ are symmetry in $y_1,y_2,...,y_8$ and $y_9,y_{10},...,y_{16}$, so the equations \eqref{eqs: dim2}, \eqref{eqs: dim3} are the ``dual" version of the equations \eqref{eqs: roth9}, \eqref{eqs: roth10}.) Recall that we have shown that $|X_{a,b}(\Fp)|\ll p^3$ in Step $4$, hence there are at most $O(p^3)$ choices of $y_2,y_3,y_5,y_7,y_8$. 

    \item For any given $y_1,y_2,y_3,y_4,y_5,y_6,y_7,y_8$, there are $p$ choices of $y_9$.

    \item For any given $y_1,y_2,y_3,y_4,y_5,y_6,y_7,y_8,y_9$, there are at most $O(1)$ choices of $y_{11},y_{10}$ by equations \eqref{eqs: roth1}, \eqref{eqs: roth5}.

    \item For any given $y_1,y_2,y_3,y_4,y_5,y_6,y_7,y_8,y_9,y_{11},y_{10}$, there are at most $O(1)$ choices of $y_{12},y_{15}$ by equations \eqref{eqs: roth2}, \eqref{eqs: roth6}.

    \item For any given $y_1,y_2,y_3,y_4,y_5,y_6,y_7,y_8,y_9,y_{11},y_{10},y_{12},y_{15}$, there are at most $O(1)$ choices of $y_{13},y_{16}$ by equations \eqref{eqs: roth3}, \eqref{eqs: roth7}.

    \item For any given $y_1,y_2,y_3,y_4,y_5,y_6,y_7,y_8,y_9,y_{11},y_{10},y_{12},y_{15},y_{13},y_{16}$, there are at most $O(1)$ choices of $y_{14}$ by equation \eqref{eqs: roth4}.

    \item Thus, there are at most $O(p^2p^3p)=O(p^6)$ choices of $y_1,y_2,y_3,y_4,y_5,y_6,y_7,y_8,y_9,\\y_{11},y_{10},y_{12},y_{15},y_{13},y_{16},y_{14}$.

\end{itemize}

    Therefore, we conclude that $|Y_{1}(\Fp)|\ll p^6$.\\

    \underline{\textbf{Step $8$: Bounding $|Y_{9,10}(\Fp)|$}}

    This step is more of a routine step and can be skipped on a first reading.

    We would define auxiliary varieties $Y_{9,10}'$ and $W_{a,b}$. Let $Y_{9,10}'$ be the variety cut out by the following three equations in eight variables $y_9,y_{10},...,y_{16}$.
    \begin{gather}
        P(y_{9})-P(y_{10})-P(y_{11})+P(y_{12})-P(y_{13})+P(y_{14})+P(y_{15})-P(y_{16})=0\label{eqs: dim7}\\
        Q(y_{9})-Q(y_{10})-Q(y_{11})+Q(y_{12})-Q(y_{13})+Q(y_{14})+Q(y_{15})-Q(y_{16})=0\label{eqs: dim8}\\
        -P'(y_9)Q'(y_{10})+Q'(y_9)P'(y_{10})=0
    \end{gather}
    For any $1\leq a,b\leq p$, let $W_{a,b}$ be the variety cut out by the following two equations in six variables $y_{11},y_{12},y_{13},y_{14},y_{15},y_{16}$.
    \begin{gather}
        -P(y_{11})+P(y_{12})-P(y_{13})+P(y_{14})+P(y_{15})-P(y_{16})=a\label{eqs: dim9}\\
        -Q(y_{11})+Q(y_{12})-Q(y_{13})+Q(y_{14})+Q(y_{15})-Q(y_{16})=b\label{eqs: dim10}
    \end{gather}
    Note that $Y'_{9,10}(\Fp)\subset \Fp^8$ and $W_{a,b}(\Fp)\subset \Fp^6$. We would like to show that $|W_{a,b}(\Fp)|\ll p^4$, and use this to show that $|Y'_{9,10}(\Fp)|\ll p^5$, and use this to show that $|Y_{9,10}(\Fp)|\ll p^6$.

    The strategy is quite similar to previous steps. First, we decompose $W_{a,b}$ into smooth part and singular part. Let $J_W(y_{15},y_{16})$ be the determinant of the last two columns of the Jacobian matrix of the defining equations of $W_{a,b}$, that is
    \begin{equation*}
        J_W(y_{15},y_{16})=
        \begin{vmatrix}
            P'(y_{15})&-P'(y_{16})\\
            Q'(y_{15})&-Q'(y_{16})
        \end{vmatrix}
        =-P'(y_{15})Q'(y_{16})+Q'(y_{15})P'(y_{16}).
    \end{equation*}

    Define
    \begin{gather*}
        W_{a,b,\mathrm{gen}}=W_{a,b}\cap \{(y_{11},y_{12},y_{13},y_{14},y_{15},y_{16})\in \mathbb{A}^6:J_W(y_{15},y_{16})\neq 0\}\\
        W_{a,b,\mathrm{sp}}=W_{a,b}\cap \{(y_{11},y_{12},y_{13},y_{14},y_{15},y_{16})\in \mathbb{A}^6:J_W(y_{15},y_{16})=0\}
    \end{gather*}

    Note that the associated complex analytic space of $W_{a,b,\mathrm{gen}}(\mathbb{C})$ is a $4$-dimensional complex manifold since the Jacobian matrices have full rank and we can apply the implicit function theorem. By similar arguments to Step $1$, we could compare various dimensions and conclude that $|W_{a,b,\mathrm{gen}}(\Fp)|\ll p^4$, where the implied constant does not depend on $a,b$. 

    On the other hand, note that $J_W(y_{15},y_{16})$ is a nonzero rational function due to linear independence of $P,Q$ and constant function $1$. For any $(y_{11},y_{12},y_{13},y_{14},y_{15},y_{16})\in W_{a,b,\mathrm{sp}}(\Fp)$, there are at most $O(p)$ choices of $y_{15},y_{16}$ by equation $J_W(y_{15},y_{16})=0$ and Lang--Weil bound \cite{MR0065218} (or Schwartz--Zippel lemma). For any given $y_{15},y_{16}$, there are $p^3$ choices of $y_{11},y_{12},y_{13}$. For any given $y_{15},y_{16},y_{11},y_{12},y_{13}$, there are at most $O(1)$ choices of $y_{14}$ by equation \eqref{eqs: dim9}. Therefore, we conclude that $|W_{a,b,\mathrm{sp}}(\Fp)|\ll p^4$ and hence $|W_{a,b}(\Fp)|\ll p^4$ where the implied constants do not depend on $a,b$.

    Next, we would show that $|Y'_{9,10}(\Fp)|\ll p^5$. Note that $-P'(y_9)Q'(y_{10})+Q'(y_9)P'(y_{10})$ is a nonzero rational function due to linear independence of $P,Q$ and constant function $1$. For any $(y_9,y_{10},y_{11},y_{12},y_{13},y_{14},y_{15},y_{16})\in Y'_{9,10}(\Fp)$, there are at most $O(p)$ choices of $y_{9},y_{10}$ by equation $-P'(y_9)Q'(y_{10})+Q'(y_9)P'(y_{10})=0$ and Lang--Weil bound \cite{MR0065218} (or Schwartz--Zippel lemma). Fix $y_9,y_{10}$, for any $(y_9,y_{10},y_{11},y_{12},y_{13},y_{14},y_{15},y_{16})\in Y'_{9,10}(\Fp)$, we have $(y_{11},y_{12},y_{13},y_{14},y_{15},y_{16})\in W_{a,b}(\Fp)$ for some $a,b$, hence there are at most $O(p^4)$ choices of $y_{11},y_{12},y_{13},y_{14},y_{15},y_{16}$. Therefore, we conclude that $|Y'_{9,10}(\Fp)|\ll p^5$.

    Finally, we would show that $|Y_{9,10}(\Fp)|\ll p^6$.
    \begin{itemize}
    \item For any $(y_1,y_2,...,y_{16})\in Y_{9,10}(\Fp)$, we have $(y_9,y_{10},...,y_{16})\in Y'_{9,10}(\Fp)$ by equations \eqref{eqs: roth9}, \eqref{eqs: roth10} and the construction of $Y'_{9,10}$, hence there are at most $O(p^5)$ choices of $y_9,y_{10},y_{11},y_{12},y_{13},y_{14},y_{15},y_{16}$.

    \item For any given $y_9,y_{10},y_{11},y_{12},y_{13},y_{14},y_{15},y_{16}$, there are $p$ choices of $y_1$.

    \item For any given $y_9,y_{10},y_{11},y_{12},y_{13},y_{14},y_{15},y_{16},y_1$, there are at most $O(1)$ choices of $y_{3},y_{2}$ by equations \eqref{eqs: roth1}, \eqref{eqs: roth5}.

    \item For any given $y_9,y_{10},y_{11},y_{12},y_{13},y_{14},y_{15},y_{16},y_1,y_{3},y_{2}$, there are at most $O(1)$ choices of $y_{4},y_{7}$ by equations \eqref{eqs: roth2}, \eqref{eqs: roth6}.

    \item For any given $y_9,y_{10},y_{11},y_{12},y_{13},y_{14},y_{15},y_{16},y_1,y_{3},y_{2},y_{4},y_{7}$, there are at most $O(1)$ choices of $y_{5},y_{8}$ by equations \eqref{eqs: roth3}, \eqref{eqs: roth7}.

    \item For any given $y_9,y_{10},y_{11},y_{12},y_{13},y_{14},y_{15},y_{16},y_1,y_{3},y_{2},y_{4},y_{7},y_{5},y_{8}$, there are at most $O(1)$ choices of $y_{6}$ by equation \eqref{eqs: roth4}.

    \item Thus, there are at most $O(p^5p)=O(p^6)$ choices of $y_9,y_{10},y_{11},y_{12},y_{13},y_{14},y_{15},y_{16},y_1,\\y_{3},y_{2},y_{4},y_{7},y_{5},y_{8},y_6$. 

\end{itemize}

    Therefore, we conclude that $|Y_{9,10}(\Fp)|\ll p^6$.
\end{proof}

\section{Degree lowering}\label{sec: degree}

In this section, we would put everything together to prove Theorem \ref{thm: main}. In Sections \ref{sec: PET} and \ref{sec: dim}, we obtain Gowers norm control for counting operator for corners generated by $P(t),Q(t)$, hence we could apply the degree lowering pioneered by Peluse \cite{peluse} to complete the proof of Theorem \ref{thm: main}. Actually, we need a different version of degree lowering that works for directional Gowers norms, and this version of degree lowering is developed by Kuca \cite{kuca-2} in his work on multidimensional Szemer\'{e}di theorem in finite fields. However, we do not need full strength of Kuca's work since we start degree lowering with $U^2$-norm as opposed to general $U^s$-norm, for this reason, and for keeping track of the power-saving exponents and the sake of completeness, we reproduce Kuca's arguments in this section. In this section, we closely follow Kuca's paper \cite[Section 6]{kuca-2}.

\begin{defn}[Eigenfunctions]\hfill
    \begin{enumerate}
        \item A function $\chi:\Fp^2\longrightarrow\C$ is an eigenfunction with respect to the first coordinate if there exists a subset $E\subset \Fp$ and phase functions $\phi:\Fp\longrightarrow\Fp,\psi:\Fp\longrightarrow\R$ such that
        \begin{equation*}
            \chi(x_1,x_2)=1_{E}(x_2)e_p(\phi(x_2)x_1)e(\psi(x_2))
        \end{equation*}
        for all $(x_1,x_2)\in \Fp^2$.
        \item A function $\chi:\Fp^2\longrightarrow\C$ is an eigenfunction with respect to the second coordinate if there exists a subset $E\subset \Fp$ and phase functions $\phi:\Fp\longrightarrow\Fp,\psi:\Fp\longrightarrow\R$ such that
        \begin{equation*}
            \chi(x_1,x_2)=1_{E}(x_1)e_p(\phi(x_1)x_2)e(\psi(x_1))
        \end{equation*}
        for all $(x_1,x_2)\in \Fp^2$.
    \end{enumerate}
\end{defn}

The following proposition could be regarded as an inverse theorem for $U^2(\Fp\times 0)$-norm and $U^2(0\times \Fp)$-norm. 

\begin{prop}\label{prop: inverse}
    Let $f:\Fp^2\longrightarrow\C$ be a $1$-bounded function.
    \begin{enumerate}
        \item Suppose $\lVert f\rVert_{U^2(\Fp\times 0)}\geq \delta$, then there exists an eigenfunction with respect to the first coordinate $\chi$ such that
        \begin{equation*}
            \E_{x_1,x_2}f(x_1,x_2)\chi(x_1,x_2)\geq \delta^4
        \end{equation*}
        and $\E_{x_1}f(x_1,x_2)\chi(x_1,x_2)\geq 0$ for all $x_2\in \Fp$.
        \item Suppose $\lVert f\rVert_{U^2(0\times \Fp)}\geq \delta$, then there exists an eigenfunction with respect to the second coordinate $\chi$ such that
        \begin{equation*}
            \E_{x_1,x_2}f(x_1,x_2)\chi(x_1,x_2)\geq \delta^4
        \end{equation*}
        and $\E_{x_2}f(x_1,x_2)\chi(x_1,x_2)\geq 0$ for all $x_1\in \Fp$.
    \end{enumerate}
\end{prop}

\begin{proof}
    The proofs of Proposition \ref{prop: inverse}.(1) and \ref{prop: inverse}.(2) are similar, hence we would only prove \ref{prop: inverse}.(2).

    For any $x_1$, let $g_{x_1}(x_2)=f(x_1,x_2)$. Note that
    \begin{equation*}
        \begin{split}
            \delta^4\leq \lVert f\rVert_{U^2(0\times \Fp)}^4&=\E_{x_1,x_2,h,k} f(x_1,x_2+h+k)\overline{f}(x_1,x_2+h)\overline{f}(x_1,x_2+k)f(x_1,x_2)\\
            &=\E_{x_1} \lVert g_{x_1}\rVert_{U^2}^4\\
            &=\E_{x_1}\lVert \widehat{g_{x_1}}\rVert_{l^4}^4.
        \end{split}
    \end{equation*}

    Observe that
    \begin{equation*}
    \begin{split}
        \lVert \widehat{g_{x_1}}\rVert_{l^4}^4=\sum_{\xi} \lvert\widehat{g_{x_1}}(\xi)\rvert^4&\leq \lVert \widehat{g_{x_1}}\rVert_{\infty}^2 \sum_{\xi} \lvert\widehat{g_{x_1}}(\xi)\rvert^2\\
        &=\lVert \widehat{g_{x_1}}\rVert_{\infty}^2\lVert g_{x_1}\rVert_{2}^2\\
        &\leq \lVert \widehat{g_{x_1}}\rVert_{\infty}.
    \end{split}        
    \end{equation*}

    Hence, for every $x_1$, there exists a frequency $\xi_{x_1}$ such that
    \begin{equation*}
        \delta^4\leq \E_{x_1}\lvert\widehat{g_{x_1}}(\xi_{x_1})\rvert=\E_{x_1}\biggl\lvert\E_{x_2}f(x_1,x_2)e_p(-\xi_{x_1} x_2)\biggr\rvert.
    \end{equation*}

    Choosing phase functions $\phi:\Fp\longrightarrow\Fp,\psi:\Fp\longrightarrow\R$ such that
    \begin{equation*}
        \biggl\lvert\E_{x_2}f(x_1,x_2)e_p(-\xi_{x_1} x_2)\biggr\rvert=\E_{x_2}f(x_1,x_2)e_p(\phi(x_1)x_2)e(\psi(x_1))
    \end{equation*}
    completes the proof. (Note that in the definition of eigenfunctions, there is a characteristic function $1_{E}$ for some subset $E$, here we can choose $E=\Fp$.)
\end{proof}

The following proposition is an asymptotic formula for two term progressions, which is needed for induction.

\begin{prop}\label{prop: 2 terms}
    Let $P(t)\in \Q(t)$ be a rational function over $\Q$.
    \begin{enumerate}
        \item For all $1$-bounded functions $f_0,f_1:\Fp^2\longrightarrow\C$, we have
        \begin{equation*}
            \E_{x_1,x_2,y}f_0(x_1,x_2)f_1(x_1+P(y),x_2)=\E_{x_1,x_2}\Bigl(f_0(x_1,x_2)\E_{a}f_1(a,x_2)\Bigr)+O\bigl(p^{-\frac{1}{2}}\bigr).
        \end{equation*}
        \item For all $1$-bounded functions $f_0,f_1:\Fp^2\longrightarrow\C$, we have
        \begin{equation*}
            \E_{x_1,x_2,y}f_0(x_1,x_2)f_1(x_1,x_2+P(y))=\E_{x_1,x_2}\Bigl(f_0(x_1,x_2)\E_{a}f_1(x_1,a)\Bigr)+O\bigl(p^{-\frac{1}{2}}\bigr).
        \end{equation*}
    \end{enumerate}
\end{prop}

\begin{proof}
    The proofs of Proposition \ref{prop: 2 terms}.(1) and \ref{prop: 2 terms}.(2) are similar, hence we would only prove Proposition \ref{prop: 2 terms}.(1).

    Note that after fixing $x_2$, the counting operator reduces to the one dimensional case, and we can invoke the results \cite[Proposition 5.1]{hong-lim} for one dimensional progressions. Hence, we have
    \begin{equation*}
        \begin{split}
            \E_{x_1,x_2,y}f_0(x_1,x_2)f_1(x_1+P(y),x_2)&=\E_{x_2}\Bigl(\E_{x_1,y}f_0(x_1,x_2)f_1(x_1+P(y),x_2)\Bigr)\\
            &=\E_{x_2}\Bigl(\E_{x_1}f_0(x_1,x_2)\E_{a}f_1(a,x_2)+O\bigl(p^{-\frac{1}{2}}\bigr)\Bigr)\\
            &=\E_{x_1,x_2}\Bigl(f_0(x_1,x_2)\E_{a}f_1(a,x_2)\Bigr)+O\bigl(p^{-\frac{1}{2}}\bigr)
        \end{split}
    \end{equation*}
\end{proof}

The following theorem contains the induction process for degree lowering. Essentially, we would induct on the number of the eigenfunctions in the counting operators. That being said, we choose to write the proof not in the exact format of an induction so that the notations could be less dense. Also, note that Theorem \ref{thm: degree}.(3) is same as Theorem \ref{thm: main}. 

\begin{thm}\label{thm: degree}
    Let $P(t),Q(t)\in \Q(t)$ be rational functions such that $P(t),Q(t)$ and the constant function $1$ are linearly independent over $\Q$, and $f_0,f_1,f_2:\Fp^2\longrightarrow\C$ be $1$-bounded functions.
    \begin{enumerate}
        \item Suppose $f_1$ is an eigenfunction with respect to the first coordinate and $f_2$ is an eigenfunction with respect to the second coordinate, then we have
        \begin{multline*}
        \E_{x_1,x_2,y} f_0(x_1,x_2) f_1(x_1+P(y),x_2) f_2(x_1,x_2+Q(y))\\
        =\E_{x_1,x_2} \Bigl(f_0(x_1,x_2) \E_{a}f_1(a,x_2) \E_{b}f_2(x_1,b)\Bigr)+O_{P,Q}\bigl(p^{-\frac{1}{2}}\bigr).
    \end{multline*}
    \item Suppose $f_2$ is an eigenfunction with respect to the second coordinate, then we have
        \begin{multline*}
        \E_{x_1,x_2,y} f_0(x_1,x_2) f_1(x_1+P(y),x_2) f_2(x_1,x_2+Q(y))\\
        =\E_{x_1,x_2} \Bigl(f_0(x_1,x_2) \E_{a}f_1(a,x_2) \E_{b}f_2(x_1,b)\Bigr)+O_{P,Q}\bigl(p^{-\frac{1}{640}}\bigr).
    \end{multline*}
    \item For all $1$-bounded functions $f_0,f_1,f_2:\Fp^2\longrightarrow\C$, we have
        \begin{multline*}
        \E_{x_1,x_2,y} f_0(x_1,x_2) f_1(x_1+P(y),x_2) f_2(x_1,x_2+Q(y))\\
        =\E_{x_1,x_2} \Bigl(f_0(x_1,x_2) \E_{a}f_1(a,x_2) \E_{b}f_2(x_1,b)\Bigr)+O_{P,Q}\bigl(p^{-\frac{1}{40960}}\bigr).
    \end{multline*}
    \end{enumerate}
\end{thm}
    
\begin{proof}[Proof of Theorem \ref{thm: degree}.(1)]
    This would be the base case for the induction for Theorem \ref{thm: degree}. Since $f_1,f_2$ are eigenfunctions, there exist subsets $E,F\subset\Fp$, phase functions $\phi:\Fp\longrightarrow\Fp,\alpha:\Fp\longrightarrow\Fp$ and $\psi:\Fp\longrightarrow\R,\beta:\Fp\longrightarrow\R$ such that
    \begin{gather*}
        f_1(x_1,x_2)=1_{E}(x_2)e_p(\phi(x_2)x_1)e(\psi(x_2))\\
        f_2(x_1,x_2)=1_{F}(x_1)e_p(\alpha(x_1)x_2)e(\beta(x_1))
    \end{gather*}
    for all $(x_1,x_2)\in \Fp^2$.

    Expanding the counting operator, we have
    \begin{equation*}
        \begin{split}
            &\E_{x_1,x_2,y} f_0(x_1,x_2) f_1(x_1+P(y),x_2) f_2(x_1,x_2+Q(y))\\
        &=\E_{x_1,x_2} f_0(x_1,x_2)1_{E}(x_2)e_p(\phi(x_2)x_1)e(\psi(x_2))1_{F}(x_1)e_p(\alpha(x_1)x_2)e(\beta(x_1))\E_{y}e_p(\phi(x_2)P(y)+\alpha(x_1)Q(y)) 
        \end{split}
    \end{equation*}

    By Bombieri's bound on single variable exponential sums in finite fields with rational function phases \cite{MR0200267} (also see \cite[Proposition 2.2]{bourgain-chang}), we have
    \begin{equation*}
        \biggl\lvert\E_{y}e_p(\phi(x_2)P(y)+\alpha(x_1)Q(y))\biggr\rvert\ll p^{-1/2}
    \end{equation*}
    when $\phi(x_2)\neq 0$ or $\alpha(x_1)\neq 0$.

    Therefore, the counting operator is equal to
    \begin{equation*}
        \E_{x_1,x_2} f_0(x_1,x_2)1_{E}(x_2)e(\psi(x_2))1_{\phi(x_2)=0}1_{F}(x_1)e(\beta(x_1))1_{\alpha(x_1)=0}+O(p^{-1/2}).
    \end{equation*}

    To complete the proof, note that the directional averages of eigenfunctions have a special form, i.e.,
    \begin{gather*}
        \E_{x_1}f_1(x_1,x_2)=\E_{x_1}1_{E}(x_2)e_p(\phi(x_2)x_1)e(\psi(x_2))=1_{E}(x_2)e(\psi(x_2))1_{\phi(x_2)=0}\\
        \E_{x_2}f_2(x_1,x_2)=\E_{x_2}1_{F}(x_1)e_p(\alpha(x_1)x_2)e(\beta(x_1))=1_{F}(x_1)e(\beta(x_1))1_{\alpha(x_1)=0}.
    \end{gather*}
\end{proof}

\begin{proof}[Proof of Theorem \ref{thm: degree}.(2)]
First, we would like to de-mean $f_2$ with respect to variable $x_2$ by splitting $f_2(x_1,x_2)=(f_2(x_1,x_2)-\E_{a}f_2(x_1,a))+\E_{a}f_2(x_1,a)$. The term $\E_{a}f_2(x_1,a)$ contributes to the main term since by Proposition \ref{prop: 2 terms}.(1), we have
\begin{equation*}
    \begin{split}
        &\E_{x_1,x_2,y} f_0(x_1,x_2) f_1(x_1+P(y),x_2)\E_{a}f_2(x_1,a)\\
        &=\E_{x_1,x_2,y} \Bigl(f_0(x_1,x_2)\E_{a}f_2(x_1,a)\Bigr) f_1(x_1+P(y),x_2)\\
        &=\E_{x_1,x_2} \Bigl(f_0(x_1,x_2)\E_{a}f_2(x_1,a)\Bigr) \E_{b}f_1(b,x_2)+O(p^{-1/2})\\
        &=\E_{x_1,x_2} \Bigl(f_0(x_1,x_2) \E_{b}f_1(b,x_2)\E_{a}f_2(x_1,a)\Bigr)+O(p^{-1/2}).
    \end{split}
\end{equation*}

By replacing $f_2(x_1,x_2)$ by $f_2(x_1,x_2)-\E_{a}f_2(x_1,a)$, we can assume the directional average of $f_2$ with respect to $x_2$ is zero, that is $\E_{x_2}f_2(x_1,x_2)=0$. (Also, note that $f_2(x_1,x_2)-\E_{a}f_2(x_1,a)$ is still an eigenfunction with respect to the second variable.)

Now, given $1$-bounded functions $f_0,f_1,f_2$ where $f_2$ is an eigenfunction with respect to the second coordinate with $\E_{x_2}f_2(x_1,x_2)=0$, let 
\begin{equation*}
    \delta=\biggl\lvert\E_{x_1,x_2,y} f_0(x_1,x_2) f_1(x_1+P(y),x_2)f_2(x_1,x_2+Q(y))\biggr\rvert
\end{equation*}
be the quantity that we would like to upper bound. Define the dual function
\begin{equation*}
    F_1(x_1,x_2)=\E_{y}f_0(x_1-P(y),x_2) f_2(x_1-P(y),x_2+Q(y)).
\end{equation*}

Note that
\begin{equation*}
    \delta=\biggl\lvert\E_{x_1,x_2,y}  f_1(x_1,x_2)F_1(x_1,x_2)\biggr\rvert\leq \biggl(\E_{x_1,x_2}\lvert F_1(x_1,x_2)\rvert^2\biggr)^{1/2}.
\end{equation*}

Expanding $\lvert F_1(x_1,x_2)\rvert^2$, we have
\begin{equation*}
    \E_{x_1,x_2,y} f_0(x_1,x_2) \overline{F_1}(x_1+P(y),x_2)f_2(x_1,x_2+Q(y))\geq \delta^2.
\end{equation*}

By algebraic geometry PET induction (Theorem \ref{thm: PETdual}) and the dimension estimates for the Roth variety (Theorem \ref{thm: dim}), we deduce that
\begin{equation*}
    \lVert F_1 \rVert_{U^2(\Fp\times 0)}^{1/4}\gg \delta^2.
\end{equation*}

By inverse theorem for $U^2(\Fp\times 0)$ (Proposition \ref{prop: inverse}), there exists an eigenfunction with respect to the first coordinate $\chi$ such that
\begin{equation*}
    \E_{x_1,x_2} F_1(x_1,x_2)\chi(x_1,x_2)\gg \delta^{32}
\end{equation*}
and $\E_{x_1} F_1(x_1,x_2)\chi(x_1,x_2)\geq 0$ for all $x_2\in \Fp$.

Let 
\begin{equation*}
    U=\{x_2\in\Fp:\E_{x_1} F_1(x_1,x_2)\chi(x_1,x_2)\gg \delta^{32}/2\},
\end{equation*}
then by the pigeonhole principle, we have $\frac{|U|}{p}\gg \delta^{32}$. Hence, we deduce that
\begin{equation*}
    \E_{x_2} 1_{U}(x_2)\E_{x_1}F_1(x_1,x_2)\chi(x_1,x_2)\gg \delta^{64},
\end{equation*}
which is 
\begin{equation*}
    \E_{x_1,x_2,y} f_0(x_1,x_2) \chi(x_1+P(y),x_2)1_{U}(x_2)f_2(x_1,x_2+Q(y))\gg \delta^{64}.
\end{equation*}

By induction on Theorem \ref{thm: degree}.(1) for the functions $f_0,\chi 1_{U},f_2$, we have
\begin{equation*}
    \E_{x_1,x_2} \Bigl(f_0(x_1,x_2) \E_{a}(\chi 1_{U})(a,x_2) \E_{b}f_2(x_1,b)\Bigr)+O(p^{-1/2})\gg \delta^{64}.
\end{equation*}

Note that the desired estimate in Theorem \ref{thm: degree}.(2) already holds when $p^{-1/2} \gg \delta^{64}$, since we are in the mean zero case. Thus, without loss of generality, we may eliminate the error term $O(p^{-1/2})$ in the above estimate. By Cauchy--Schwarz inequality in variable $x_2$, we have
\begin{equation*}
    \E_{x_2} \biggl\lvert  \E_{a}(\chi 1_{U})(a,x_2) \biggr\rvert^2\gg \delta^{128}.
\end{equation*}

Since $\chi$ is an eigenfunction with respect to the first coordinate, there exists a subset $E\subset \Fp$, phase functions $\phi:\Fp\longrightarrow\Fp,\psi:\Fp\longrightarrow\R$ such that $\chi(x_1,x_2)=1_{E}(x_2)e_p(\phi(x_2)x_1)e(\psi(x_2))$. Therefore, we have
\begin{equation*}
    \E_{x_2} \biggl\lvert  \E_{a}1_{E}(x_2)e_p(\phi(x_2)a)e(\psi(x_2))1_{U}(x_2) \biggr\rvert^2\gg \delta^{128},
\end{equation*}
which is
\begin{equation*}
    \E_{x_2} \biggl\lvert  1_{E}(x_2)1_{U}(x_2)1_{\phi(x_2)=0} \biggr\rvert^2\gg \delta^{128}.
\end{equation*}

By definition of $U$, we deduce that
\begin{equation*}
    \E_{x_2}  1_{E}(x_2)1_{U}(x_2)1_{\phi(x_2)=0}\E_{x_1}F_1(x_1,x_2)\chi(x_1,x_2) \gg \delta^{160},
\end{equation*}
which is 
\begin{equation*}
    \E_{x_2}  1_{E}(x_2)1_{U}(x_2)1_{\phi(x_2)=0}e(\psi(x_2))\E_{x_1}F_1(x_1,x_2) \gg \delta^{160}.
\end{equation*}

Applying Cauchy--Schwarz inequality in $x_2$, we have successfully passed from $U^2(\Fp\times0)$-norm to $U^1(\Fp\times0)$-norm and lowered the degree, i.e.,
\begin{equation*}
    \delta^{320}\ll \E_{x_2}\biggl\lvert\E_{x_1}F_1(x_1,x_2)\biggr\rvert^2=\lVert F_1 \rVert_{U^1(\Fp\times0)}^2
\end{equation*}

Finally, expanding the squares $\lvert\E_{x_1}F_1(x_1,x_2)\rvert^2$, we have
\begin{equation*}
    \E_{x_1,x_2,y} \Bigl(f_0(x_1,x_2)f_2(x_1,x_2+Q(y)) \E_{a}\overline{F_1}(a,x_2)\Bigr)\gg \delta^{320},
\end{equation*}

By the asymptotic formula for two term progressions (Proposition \ref{prop: 2 terms}.(2)) for the functions $f_0(x_1,x_2) \E_{a}\overline{F_1}(a,x_2)$ and $f_2(x_1,x_2)$, we conclude
\begin{equation*}
    \E_{x_1,x_2,y} \Bigl(f_0(x_1,x_2) \E_{a}\overline{F_1}(a,x_2)\E_{b}f_2(x_1,b)\Bigr)+O(p^{-1/2})\gg \delta^{320}
\end{equation*}
The main term is zero since we already eliminated the mean of $f_2$ with respect to variable $x_2$ in the beginning of the argument, hence we have $\delta^{320}\ll p^{-1/2}$ which is equivalent to $\delta\ll p^{-1/640}$.  

\end{proof}

\begin{proof}[Proof of Theorem \ref{thm: degree}.(3) and Theorem \ref{thm: main}]
The proof of Theorem \ref{thm: degree}.(3) is similar to the proof of Theorem \ref{thm: degree}.(2), hence we would only briefly indicate the changes. Basically, we reverse the roles of $x_1$ and $x_2$ and induct on Theorem \ref{thm: degree}.(2). Here is a list of the changes.
\begin{itemize}
    \item We would de-mean $f_1$ with respect to $x_1$.
    \item The dual function would be $F_2(x_1,x_2)=\E_{y}f_0(x_1,x_2-Q(y))f_1(x_1+P(y),x_2-Q(y))$.
    \item We have $\lVert F_2 \rVert_{U^2(0\times \Fp)}$ instead of $\lVert F_1 \rVert_{U^2(\Fp\times 0)}$ since we would invoke Theorem \ref{thm: PET}.
    \item We would invoke Proposition \ref{prop: inverse}.(2) to get an eigenfunction with respect to the second coordinate $\chi$.
    \item The definition of $U$ would be $\{x_1\in\Fp:\E_{x_2} F_2(x_1,x_2)\chi(x_1,x_2)\gg \delta^{32}/2\}$.
    \item We would induct on Theorem \ref{thm: degree}.(2) instead of Theorem \ref{thm: degree}.(1) for the averages
    \begin{equation*}
        \E_{x_1,x_2,y} f_0(x_1,x_2) f_1(x_1+P(y),x_2)\chi(x_1,x_2+Q(y))1_{U}(x_1)\gg \delta^{64},
    \end{equation*}
    hence we would get an error term $O(p^{-1/640})$ instead of $O(p^{-1/2})$, this would lead to the bound $\delta\ll p^{-1/(640\cdot64)}=p^{-1/(40960)}$.
\end{itemize}
\end{proof}

\printbibliography 

@misc{kavrut-wu,
      title={A bilinear estimate in $\mathbb{F}_p$}, 
      author={Necef Kavrut and Shukun Wu},
      year={2024},
      eprint={2401.07925},
      archivePrefix={arXiv}, 
}

@book {MR0463157,
    AUTHOR = {Hartshorne, Robin},
     TITLE = {Algebraic geometry},
    SERIES = {Graduate Texts in Mathematics},
    VOLUME = {No. 52},
 PUBLISHER = {Springer-Verlag, New York-Heidelberg},
      YEAR = {1977},
     PAGES = {xvi+496},
   MRCLASS = {14-01},
  MRNUMBER = {463157},
MRREVIEWER = {Robert\ Speiser},
}

@book {MR0755331,
    AUTHOR = {Grauert, Hans and Remmert, Reinhold},
     TITLE = {Coherent analytic sheaves},
    SERIES = {Grundlehren der mathematischen Wissenschaften [Fundamental
              Principles of Mathematical Sciences]},
    VOLUME = {265},
 PUBLISHER = {Springer-Verlag, Berlin},
      YEAR = {1984},
     PAGES = {xviii+249},
   MRCLASS = {32-02 (32Bxx 32C30)},
  MRNUMBER = {755331},
MRREVIEWER = {Daniel\ Barlet},
}

@misc{stacks-project,
  author       = {The {Stacks project authors}},
  title        = {The Stacks project},
  howpublished = {\url{https://stacks.math.columbia.edu}},
}

@misc{Tao,
  author       = {Tao, Terence},
  title        = {The Lang-Weil Bound},
  howpublished = {\url{https://terrytao.wordpress.com/2012/08/31/the-lang-weil-bound/}},
}

@article {MR0065218,
    AUTHOR = {Lang, Serge and Weil, Andr\'{e}},
     TITLE = {Number of points of varieties in finite fields},
   JOURNAL = {Amer. J. Math.},
  FJOURNAL = {American Journal of Mathematics},
    VOLUME = {76},
      YEAR = {1954},
     PAGES = {819--827},
   MRCLASS = {14.0X},
  MRNUMBER = {65218},
MRREVIEWER = {B.\ Segre},
}

@misc{hong-lim,
      title={Three term rational function progressions in finite fields}, 
      author={Guo-Dong Hong and Zi Li Lim},
      year={2024},
      eprint={2401.01137},
      archivePrefix={arXiv}, 
}

@article {bourgain-chang,
    AUTHOR = {Bourgain, J. and Chang, M.-C.},
     TITLE = {Nonlinear {R}oth type theorems in finite fields},
   JOURNAL = {Israel J. Math.},
  FJOURNAL = {Israel Journal of Mathematics},
    VOLUME = {221},
      YEAR = {2017},
    NUMBER = {2},
     PAGES = {853--867},
   MRCLASS = {11B30 (37A45)},
  MRNUMBER = {3704938},
MRREVIEWER = {Ben\ Joseph\ Green},
}

@article {MR0200267,
    AUTHOR = {Bombieri, Enrico},
     TITLE = {On exponential sums in finite fields},
   JOURNAL = {Amer. J. Math.},
  FJOURNAL = {American Journal of Mathematics},
    VOLUME = {88},
      YEAR = {1966},
     PAGES = {71--105},
   MRCLASS = {14.48 (10.41)},
  MRNUMBER = {200267},
MRREVIEWER = {D.\ J.\ Lewis},
}

@article {han-lacey-yang,
    AUTHOR = {Han, Rui and Lacey, Michael T. and Yang, Fan},
     TITLE = {A polynomial {R}oth theorem for corners in finite fields},
   JOURNAL = {Mathematika},
  FJOURNAL = {Mathematika. A Journal of Pure and Applied Mathematics},
    VOLUME = {67},
      YEAR = {2021},
    NUMBER = {4},
     PAGES = {885--896},
   MRCLASS = {11B30},
  MRNUMBER = {4304416},
MRREVIEWER = {Max\ Wenqiang\ Xu},
}

@article {peluse,
    AUTHOR = {Peluse, Sarah},
     TITLE = {On the polynomial {S}zemer\'edi theorem in finite fields},
   JOURNAL = {Duke Math. J.},
  FJOURNAL = {Duke Mathematical Journal},
    VOLUME = {168},
      YEAR = {2019},
    NUMBER = {5},
     PAGES = {749--774},
   MRCLASS = {11B30 (11B25)},
  MRNUMBER = {3934588},
MRREVIEWER = {Pierre-Yves\ Bienvenu},
}

@article {kuca-2,
    AUTHOR = {Kuca, Borys},
     TITLE = {Multidimensional polynomial patterns over finite fields:
              bounds, counting estimates and {G}owers norm control},
   JOURNAL = {Adv. Math.},
  FJOURNAL = {Advances in Mathematics},
    VOLUME = {448},
      YEAR = {2024},
     PAGES = {Paper No. 109700, 61},
   MRCLASS = {11B30},
  MRNUMBER = {4741361},
}

@article {bergelson-leibman,
    AUTHOR = {Bergelson, V. and Leibman, A.},
     TITLE = {Polynomial extensions of van der {W}aerden's and
              {S}zemer\'edi's theorems},
   JOURNAL = {J. Amer. Math. Soc.},
  FJOURNAL = {Journal of the American Mathematical Society},
    VOLUME = {9},
      YEAR = {1996},
    NUMBER = {3},
     PAGES = {725--753},
   MRCLASS = {11B25 (05D10 28D05 54H20)},
  MRNUMBER = {1325795},
MRREVIEWER = {Pierre\ Michel},
}

@article {peluse-polynomial,
    AUTHOR = {Peluse, Sarah},
     TITLE = {Bounds for sets with no polynomial progressions},
   JOURNAL = {Forum Math. Pi},
  FJOURNAL = {Forum of Mathematics. Pi},
    VOLUME = {8},
      YEAR = {2020},
     PAGES = {e16, 55},
   MRCLASS = {11B30 (11B25)},
  MRNUMBER = {4199235},
MRREVIEWER = {Ben\ Joseph\ Green},
}

@article {peluse-prendiville,
    AUTHOR = {Peluse, S. and Prendiville, S.},
     TITLE = {Quantitative bounds in the nonlinear Roth theorem},
   JOURNAL = {Invent. math.},
    VOLUME = {238},
      YEAR = {2024},
     PAGES = {865–903},
}

@article {peluse-L,
    AUTHOR = {Peluse, Sarah},
     TITLE = {Subsets of {$\Bbb F^n_p \times \Bbb F^n_p$} without {$\rm
              L$}-shaped configurations},
   JOURNAL = {Compos. Math.},
  FJOURNAL = {Compositio Mathematica},
    VOLUME = {160},
      YEAR = {2024},
    NUMBER = {1},
     PAGES = {176--236},
   MRCLASS = {11B30 (05D99)},
  MRNUMBER = {4676193},
}

@misc{austin,
      title={Partial difference equations over compact Abelian groups, I: modules of solutions}, 
      author={Tim Austin},
      year={2014},
      eprint={1305.7269},
      archivePrefix={arXiv},
}

@article {kuca-1,
    AUTHOR = {Kuca, Borys},
     TITLE = {Multidimensional polynomial {S}zemer\'edi theorem in finite
              fields for polynomials of distinct degrees},
   JOURNAL = {Israel J. Math.},
  FJOURNAL = {Israel Journal of Mathematics},
    VOLUME = {259},
      YEAR = {2024},
    NUMBER = {2},
     PAGES = {589--620},
   MRCLASS = {11B30 (37A30)},
  MRNUMBER = {4732975},
}

@article {dong-li-sawin,
    AUTHOR = {Dong, Dong and Li, Xiaochun and Sawin, Will},
     TITLE = {Improved estimates for polynomial {R}oth type theorems in
              finite fields},
   JOURNAL = {J. Anal. Math.},
  FJOURNAL = {Journal d'Analyse Math\'{e}matique},
    VOLUME = {141},
      YEAR = {2020},
    NUMBER = {2},
     PAGES = {689--705},
   MRCLASS = {11B30},
  MRNUMBER = {4179774},
MRREVIEWER = {Robert\ F.\ Tichy},
}

@article {peluse-three-term,
    AUTHOR = {Peluse, Sarah},
     TITLE = {Three-term polynomial progressions in subsets of finite
              fields},
   JOURNAL = {Israel J. Math.},
  FJOURNAL = {Israel Journal of Mathematics},
    VOLUME = {228},
      YEAR = {2018},
    NUMBER = {1},
     PAGES = {379--405},
   MRCLASS = {11B30 (05B99)},
  MRNUMBER = {3874848},
MRREVIEWER = {Wolfgang\ A.\ Schmid},
}

@misc{peluse-prendiville-shao,
      title={Bounds in a popular multidimensional nonlinear Roth theorem}, 
      author={Sarah Peluse and Sean Prendiville and Xuancheng Shao},
      year={2024},
      eprint={2407.08338},
      archivePrefix={arXiv},
}

@misc{kravitz-kuca-leng,
      title={Corners with polynomial side length}, 
      author={Noah Kravitz and Borys Kuca and James Leng},
      year={2024},
      eprint={2407.08637},
      archivePrefix={arXiv},
}

\end{document}